\pdfoutput=1
\documentclass[a4paper,12pt,reqno,oneside]{amsart}
\usepackage[utf8x]{inputenc}
\usepackage[T1]{fontenc}
\usepackage{geometry}
\usepackage{lmodern}
\usepackage[stretch=10]{microtype}

\usepackage{amsmath}
\usepackage{amssymb}
\usepackage{amsthm}
\usepackage[matrix,arrow,cmtip]{xy}
\usepackage{enumerate}
\PassOptionsToPackage{hyphens}{url}
\usepackage[bookmarksnumbered,colorlinks,linkcolor=black,citecolor=black,urlcolor=black]{hyperref}
\usepackage[capitalise]{cleveref}

\usepackage{ifthen}
\newcounter{constant}
\newcommand{\newC}[1]{%
   \refstepcounter{constant} C_{\theconstant}%
   \ifthenelse{\equal{#1}{*}} { } {%
      \label{C:#1}%
   }%
}
\newcommand{\refC}[1]{C_{\ref*{C:#1}}}

\newtheorem{lemma}{Lemma}[section]
\newtheorem{proposition}[lemma]{Proposition}
\newtheorem{theorem}[lemma]{Theorem}

\newtheorem{conjecture}[lemma]{Conjecture}
\Crefname{conjecture}{Conjecture}{Conjectures} 
\newtheorem*{theorem*}{Theorem}

\theoremstyle{definition}
\newtheorem{remark}[lemma]{Remark}

\theoremstyle{definition}
\newtheorem*{definition}{Definition}

\newcommand{\bA}{\mathbb{A}}
\newcommand{\AAA}{\mathbb{A}}
\newcommand{\bC}{\mathbb{C}}
\newcommand{\CC}{\mathbb{C}}
\newcommand{\bG}{\mathbb{G}}
\newcommand{\bQ}{\mathbb{Q}}
\newcommand{\QQ}{\mathbb{Q}}
\newcommand{\bR}{\mathbb{R}}
\newcommand{\RR}{\mathbb{R}}
\newcommand{\bS}{\mathbb{S}}
\newcommand{\bZ}{\mathbb{Z}}
\newcommand{\ZZ}{\mathbb{Z}}
\newcommand{\Qbar}{\overline \bQ}

\newcommand{\cA}{\mathcal{A}}
\newcommand{\cF}{\mathcal{F}}
\newcommand{\cH}{\mathcal{H}}
\newcommand{\cO}{\mathcal{O}}
\newcommand{\Ag}{\mathcal{A}_g}
\newcommand{\Agt}{\cA_{g,3}}
\newcommand{\Hg}{\mathcal{H}_g}

\newcommand{\fA}{\mathfrak{A}}
\newcommand{\fF}{\mathfrak{F}}
\newcommand{\fG}{\mathfrak{G}}
\newcommand{\fS}{\mathfrak{S}}

\newcommand{\gG}{\mathbf{G}}
\newcommand{\gGL}{\mathbf{GL}}
\newcommand{\gGSp}{\mathbf{GSp}}
\newcommand{\gH}{\mathbf{H}}
\newcommand{\gM}{\mathbf{M}}
\newcommand{\gP}{\mathbf{P}}
\newcommand{\gS}{\mathbf{S}}
\newcommand{\gSL}{\mathbf{SL}}
\newcommand{\gSO}{\mathbf{SO}}
\newcommand{\gSp}{\mathbf{Sp}}
\newcommand{\gT}{\mathbf{T}}
\newcommand{\gU}{\mathbf{U}}

\newcommand{\rH}{\mathrm{H}}
\newcommand{\rM}{\mathrm{M}}
\newcommand{\rO}{\mathrm{O}}

\DeclareMathOperator{\Aut}{Aut}
\DeclareMathOperator{\diam}{diam}
\DeclareMathOperator{\disc}{disc}
\DeclareMathOperator{\End}{End}
\DeclareMathOperator{\Gal}{Gal}
\DeclareMathOperator{\Hom}{Hom}
\DeclareMathOperator{\lcm}{lcm}
\DeclareMathOperator{\MT}{MT}
\DeclareMathOperator{\Res}{Res}
\DeclareMathOperator{\rk}{rk}
\DeclareMathOperator{\Sh}{Sh}
\DeclareMathOperator{\Sp}{Sp}
\DeclareMathOperator{\Stab}{Stab}
\DeclareMathOperator{\tr}{tr}

\newcommand{\ad}{\mathrm{ad}}
\newcommand{\Cay}{\mathrm{Cay}}
\newcommand{\der}{\mathrm{der}}
\newcommand{\detstar}{\det\nolimits^*}

\newcommand{\abs}[1]{\lvert #1 \rvert}

\newcommand{\bs}{\backslash}
\newcommand{\isom}{\cong}
\newcommand{\ov}{\overline}

\newcommand{\half}{\tfrac{1}{2}}

\newcommand{\fullsmallmatrix}[4]{\bigl( \begin{smallmatrix} #1 & #2 \\ #3 & #4 \end{smallmatrix} \bigr)}

\newcommand{\defterm}[1]{\textbf{#1}}

\title{Unlikely intersections with \( E \, \times \!\! \) CM curves in $\mathcal{A}_2$}

\author{Christopher Daw}
\address{Daw: Department of Mathematics and Statistics, University of Reading,
    White\-knights,  PO Box 217,  Reading,  Berkshire RG6 6AH,  United Kingdom}
\email{chris.daw@reading.ac.uk}

\author{Martin Orr}
\address{Orr: Mathematics Institute, Zeeman Building, University of Warwick, Coventry CV4 7AL, United Kingdom}
\email{martin.orr@warwick.ac.uk}


\begin{document}

\begin{abstract}
The Zilber--Pink conjecture predicts that an algebraic curve in $\mathcal{A}_2$ has only finitely many intersections with the special curves, unless it is contained in a proper special subvariety.
Under a large Galois orbits conjecture, we prove the finiteness of the intersection with the special curves parametrising abelian surfaces isogenous to the product of two elliptic curves, at least one of which has complex multiplication. Furthermore, we show that this large Galois orbits conjecture holds for curves satisfying a condition on their intersection with the boundary of the Baily--Borel compactification of $\mathcal{A}_2$. 

More generally, we show that a Hodge generic curve in an arbitrary Shimura variety has only finitely many intersection points with the generic points of a Hecke--facteur family, again under a large Galois orbits conjecture.
\end{abstract}

\maketitle

\section{Introduction}

Let \( \cA_2 \) denote the moduli space of principally polarised abelian surfaces over~\( \bC \).
Let \( V \subset \cA_2 \) be an irreducible algebraic curve.
The Zilber--Pink conjecture predicts that, if \( V \) is not contained in any proper special subvariety of \( \cA_2 \), then it should have only finitely many intersections with the special curves of \( \cA_2 \) (since \( \dim \cA_2 = 3 > 1 + 1 \)).

The special curves in \( \cA_2 \) are of three types:
\begin{enumerate}
\item curves parametrising abelian surfaces with quaternionic multiplication;
\item curves parametrising abelian surfaces isogenous to the square of an elliptic curve;
\item curves parametrising abelian surfaces isogenous to the product of two elliptic curves, at least one of which has complex multiplication (CM).
\end{enumerate}
In this paper, we consider the special curves of type (3), which we refer to as \( E \, \times \! \) CM curves. We consider the special curves of type (1) and type (2) in a forthcoming paper \cite{DO20}.

There is a general strategy for proving conjectures of Zilber--Pink type known as the Pila--Zannier strategy. This was expounded for arbitrary Shimura varieties in \cite{DR}. The upshot of this strategy is that problems of unlikely intersections can be resolved whenever two arithmetic ingredients can be obtained. The first, usually referred to as a large Galois orbits conjecture, concerns the sizes of Galois orbits of points arising from unlikely intersections. This is well-known to be a very difficult problem in general. The second ingredient, which is less well-known, concerns the parametrisation of special subvarieties and has much to do with the reduction theory of arithmetic groups.
 
In this article, we make progress on both these arithmetic ingredients for unlikely intersections with \( E \, \times \! \) CM curves.
We prove that the large Galois orbits conjecture for intersections between a general curve \( V \subset \cA_2 \) and \( E \, \times \! \) CM curves holds whenever \( V \) is defined over \( \Qbar \) and the closure of \( V \) in the Baily--Borel compactification of \( \cA_2 \) intersects the zero-dimensional stratum of the boundary. This geometric condition on $V$ can also be understood in terms of abelian schemes (see \cref{mult-reduction}).
It should be noted that this condition is not generic.
It is central to the theorem of André on which our proof of the large Galois orbits conjecture is based \cite[Ch.~X, 4.2~(i)]{And89}.

Furthermore, we solve the parametrisation problem for \( E \, \times \! \) CM curves in \( \cA_2 \) in full.
Consequently, we prove that if \( V \) is a curve not contained in any proper special subvariety of \( \cA_2 \) then it has only finitely many intersections with \( E \, \times \! \) CM curves, provided either that \( V \) satisfies the condition on its intersection with the boundary or that the large Galois orbits conjecture (\cref{galois-orbits-conj}) holds.

\subsection{Statement of main theorems}

Recall that the Baily--Borel (or minimal) compactification of $\cA_2$ has a stratification $\cA_2\sqcup\cA_1\sqcup\cA_0$ by Zariski locally closed subvarieties, where $\cA_1$ is the moduli space of elliptic curves, and $\cA_0$ is a point (see \cite[Ch.~IV, sec.~2]{FC90}). 
Our main theorem is as follows. 

\begin{theorem} \label{galois-orbits-thm}
Let \( \Sigma \) denote the set of points in \( \cA_2 \) for which the associated principally polarised abelian surface is isogenous to a product of elliptic curves \( E_1 \times E_2 \) such that \( E_2 \) has complex multiplication and $E_1$ does not.

Let \( V \) be an irreducible algebraic curve in \( \cA_2 \) which is not contained in any proper special subvariety. 

Suppose that \( V \) is defined over \( \Qbar \) and that the Zariski closure of \( V \) in the Baily--Borel compactification of \( \cA_2 \) intersects the \( 0 \)-dimensional stratum of the boundary of the compactification.

Then \( V \cap \Sigma \) is finite.
\end{theorem}

More generally, without the geometric assumption on $V$, we have the following conditional result.

\begin{theorem} \label{main-thm}
Let \( \Sigma \) be as in \cref{galois-orbits-thm}.
Let \( V \) be an irreducible algebraic curve in \( \cA_2 \) which is not contained in any proper special subvariety.

Assume that \( V \) satisfies \cref{galois-orbits-conj}.

Then \( V \cap \Sigma \) is finite.
\end{theorem}

The following large Galois orbits conjecture, referred to in the above, is similar to a special case of \cite[Conjecture~11.1]{DR} (which was itself inspired by \cite[Conjecture~8.2]{HP16}) except that we use a definition of complexity adapted to \( E \, \times \! \) CM curves in place of the more general definition used in \cite{DR}.
For the definition of the complexity \( \Delta(Z) \), see section~\ref{sec:complexity}.

\begin{conjecture} \label{galois-orbits-conj}
Let \( V \) and \( \Sigma \) be as in \cref{main-thm}. Let \( L \) be a finitely generated subfield of \( \bC \) over which \( V \) is defined.

There exist positive constants \( \newC{galois-orbits-mult} \) and \( \newC{galois-orbits-exp} \) such that, for all points \( s \in V \cap \Sigma \), if we let \( Z \) denote the (unique) special curve containing \( s \), then
\[ \# \Aut(\bC/L) \cdot s  \geq  \refC{galois-orbits-mult} \Delta(Z)^{\refC{galois-orbits-exp}}. \]
\end{conjecture}

The statement of Conjecture \ref{galois-orbits-conj} relies on the fact that each point of \( \Sigma \) lies in a \textit{unique} special curve. This is true because the intersection of two distinct special curves consists of special points and we have excluded these from \( \Sigma \). In other words, the set \( \Sigma \) is not equal to the union of the \( E \, \times \! \) CM curves, because we have removed the (countably many) points where the corresponding abelian surface is isogenous to a product of two CM elliptic curves.

Removing these special points is important for \cref{galois-orbits-conj}, because such a point lies in the intersection of infinitely many \( E \, \times \! \) CM curves.
Therefore we cannot expect the Galois orbit of such a point to be bounded below by the complexities of each of the \( E \, \times \! \) CM curves which contain it.
Instead, the Galois orbit of such a point is controlled by its own complexity as a special subvariety, as proved independently by Tsimerman \cite{Tsi12} and by Ullmo and Yafaev \cite{UY15}.

Furthermore, Pila and Tsimerman have used this bound for special points to prove the Andr\'e--Oort conjecture for $\mathcal{A}_2$ \cite{PT13}, that is, a non-special curve $V$ in~$\mathcal{A}_2$ contains only finitely many special points.
Hence, a curve $V$ as in Theorem \ref{main-thm} has finitely many intersections with the \( E \, \times \! \) CM curves if and only if it contains only finitely many points belonging to $\Sigma$.

\subsection{Alternative statements in terms of abelian schemes} \label{sec:alt-state}

It is also possible to state \cref{galois-orbits-thm,main-thm} using the language of abelian schemes instead of curves in $\cA_2$.
In these statements, if \( V \) is an algebraic curve and \( \fA \to V \) is an abelian scheme of relative dimension~\( 2 \), we say that a point \( s \in V \) is an \defterm{E \( \! \times \!\! \) CM point} if the fibre \( \fA_s \) is isogenous to a product of elliptic curves \( E_1 \times E_2 \) where \( E_2 \) has complex multiplication and \( E_1 \) does not.

Our proof of \cref{galois-orbits-thm} passes through the following theorem (we prove at \cref{go-sch-implies-go} that this theorem implies \cref{galois-orbits-thm}).

\begin{theorem} \label{galois-orbits-thm-sch}
Let \( V \) be an irreducible algebraic curve over \( \bC \) and let \( \fA \to V \) be a principally polarised non-isotrivial abelian scheme of relative dimension~\( 2 \). Suppose that \( \End(\fA_{\ov\eta}) = \bZ \), where \( \ov\eta \) denotes a geometric generic point of \( V \).

Suppose that \( V \) and \( \fA \) are defined over $\Qbar$.

Suppose that there exist a curve \( V' \), a semiabelian scheme \( \fA' \to V' \) and an open immersion \( \iota \colon V \to V' \) (all defined over \( \Qbar \)) such that \( \fA \isom \iota^* \fA' \) and there is some point \( s_0 \in V'(\ov\bQ) \) for which the fibre \( \fA'_{s_0} \) is a torus.

Then \( V \) contains only finitely many \( E \, \times \! \) CM points.
\end{theorem}

\Cref{main-thm} can be reformulated in a similar style to \cref{galois-orbits-thm-sch}.
Likewise, \cref{galois-orbits-conj} is equivalent to the following conjecture (see \cref{galois-orbits-equiv}).
The complexity $\Delta'(s)$ used in this conjecture is defined in \cref{delta's}.

\begin{conjecture} \label{galois-orbits-conj-sch}
Let \( V \) be an irreducible algebraic curve over \( \bC \) and let \( \fA \to V \) be a principally polarised non-isotrivial abelian scheme of relative dimension~\( 2 \).
Suppose that \( \End(\fA_{\ov\eta}) = \bZ \), where \( \ov\eta \) denotes a geometric generic point of \( V \).

Let \( L \) be a finitely generated subfield of \( \bC \) over which \( V \) and \( \fA \to V \) are defined.

There exist positive constants \( \newC{galois-orbits-sch-mult} \) and \( \newC{galois-orbits-sch-exp} \) such that, for all \( E \, \times \! \) CM points \( s \in V \),
\[ \# \Aut(\bC/L) \cdot s  \geq  \refC{galois-orbits-sch-mult} \Delta'(s)^{\refC{galois-orbits-sch-exp}}. \]
\end{conjecture}

\subsection{Previous results}

Of course, Unlikely Intersections is now a vast area of research. This article can be viewed as a sequel to two previous works. The first \cite{DR}, due to Ren and the first author, established a general strategy to attack the Zilber--Pink conjecture for general Shimura varieties, depending on certain arithmetic conjectures. The second \cite{OrrUI}, due to the second author, proved both conditional and unconditional results for unlikely intersections with Hecke translates of a fixed special subvariety. Both of these works were inspired by the earlier works \cite{HP12} and \cite{HP16} of Habegger and Pila on the Zilber--Pink conjecture in a product of modular curves. In this article, we implement the stategy of \cite{DR} for so-called \( E \, \times \! \) CM curves in $\mathcal{A}_2$ and, more generally, for so-called Hecke--facteur families, which are a natural generalisation of the objects studied in \cite{OrrUI}.

Considering special points instead of intersections with special curves, the Andr\'e--Oort conjecture for $\mathcal{A}_2$ was proved by Pila and Tsimerman \cite{PT13}, also using the Pila--Zannier strategy.
This strategy eventually led to their proof of the Andr\'e--Oort conjecture for $\mathcal{A}_g$ \cite{PT:Ag}, \cite{Tsim:AO}.

The key step in proving \cref{main-thm} consists in controlling the heights of algebraic points in definable sets which parametrise intersections with special subvarieties.
Analogous bounds for pre-special points of $\Ag$ appeared in \cite{PT13}, and were generalised by the authors of the current paper to arbitrary Shimura varieties \cite{DO16}.
The second author has proved bounds of a similar nature concerning Hecke operators \cite{Orr18}.
We use both of these previous bounds in this paper.

\subsection{Generalisations}

We prove a generalisation of \cref{main-thm} using the notion of Hecke--facteur families in a general Shimura variety, which we define in section~\ref{sec:hecke-facteur}.
These families are a natural generalisation of \( E \, \times \! \) CM curves in \( \cA_2 \).

For an example of Hecke--facteur families, fix positive integers \( d, e, g \) satisfying \( d+e=g \).
Consider the locus of points in \( \Ag \) parametrising principally polarised abelian varieties which are isogenous (via a polarised isogeny) to a product \( A_1 \times A_2 \), where \( \dim(A_1)=d \), \( \dim(A_2)=e \) and \( A_2 \) has CM type.
This locus is the union of a Hecke--facteur family of special subvarieties of \( \Ag \), as described in section~\ref{sec:hf-examples}.
The example includes the case of \( E \, \times \! \) CM curves by choosing \( d=e=1 \).
(Note that the definition of \( E \, \times \! \) CM curves does not require the isogeny to be polarised.
\Cref{excm-polarised} shows that this does not matter in the \( E \, \times \! \) CM case.)

Another example of a Hecke--facteur family is given by the Hecke translates of a fixed special subvariety
(this is obtained by setting \( \gH_2 = \{ 1 \} \) in the definition of Hecke--facteur family).
This case of \cref{main-thm-gen} was essentially already established in \cite[Theorem~3.4]{OrrUI}.
We are able to remove the dependence on \cite[Conjecture~3.3]{OrrUI} thanks to \cref{galoisonhecke}.

\Cref{main-thm-gen} is conditional on a Galois orbits conjecture, \cref{galois-orbits-conj-gen}.
We do not prove any cases of \cref{galois-orbits-conj-gen} beyond the \( E \, \times \! \) CM case.

\begin{theorem} \label{main-thm-gen}
Let \( (\gG, X^+) \) be a Shimura datum component and let \( S = \Gamma \bs X^+ \) be an associated Shimura variety component.
Let $V$ be an irreducible algebraic curve in $S$ not contained in any proper special subvariety.

Let \( (\gH, X_\gH^+, \gH_1, \gH_2) \) be a facteur datum in \( (\gG, X^+) \) and let \( \fF \) be the associated Hecke--facteur family of special subvarieties of \( S \).
Suppose that, for every subvariety $Z \in \fF$, we have
\[\dim(Z)\leq\dim(S)-2.\]
Let \( \Sigma \) be the set of points \( s \in S \) for which the smallest special subvariety containing \( s \) is a member of $\fF$.

Assume that \( S\), $\fF$, $V$ and $\Sigma$ satisfy \cref{galois-orbits-conj-gen}.

Then \( V \cap \Sigma \) is finite.
\end{theorem}

The following large Galois orbits conjecture gives a lower bound for the Galois degree of a point \(s \in V \cap \Sigma \) with respect to the complexity of the special subvariety of \( \fF \) containing \( s \).
The complexity \( \Delta(Z) \) of a special subvariety in the Hecke--facteur family \( \fF \) is defined in section~\ref{sec:complexity}, depending on the choice of a representation~\( \rho \) (hence why we need to make such a choice in the statement of the conjecture).

\begin{conjecture} \label{galois-orbits-conj-gen}
Let \( (\gG, X^+) \), \( S \), $\fF$, \( V \) and \( \Sigma \) be as in \cref{main-thm-gen}.
Let \( L \) be a finitely generated subfield of \( \bC \) over which \( S \) and \( V \) are defined.
Let \( \rho \colon \gG \to \gGL_{m,\bQ} \) be a faithful representation.

There exist positive constants \( \newC{galois-orbits-gen-mult} \) and \( \newC{galois-orbits-gen-exp} \) such that, for all points \( s \in V \cap \Sigma \), if we let \( Z \) denote the (unique) special subvariety in \( \fF \) containing \( s \), then
\[ \# \Aut(\bC/L) \cdot s  \geq  \refC{galois-orbits-gen-mult} \Delta(Z)^{\refC{galois-orbits-gen-exp}}. \]
\end{conjecture}

\subsection{Strategy}

The strategy to prove \cref{main-thm} is to prove \cite[Conjecture~12.2]{DR} for the pre-images of the points in $\Sigma$ (modified to use a different definition of complexity). This is achieved by proving \cite[Conjecture~12.7]{DR} for these points, which is deduced from a new height bound for representatives in congruence subgroups (\cref{lift-bound}), and combining it with previous results of the authors (\cite[Theorems 1.1 and~4.1]{DO16} and \cite[Theorem~1.1]{Orr18}). \cref{main-thm} then follows by the method of \cite[Theorem~14.2]{DR} (note that \cite[Conjecture~10.3]{DR} follows from \cite[Conjecture~12.2]{DR} by \cite[Lemma~12.5]{DR}).

Strictly speaking, we cannot directly invoke \cite[Theorem~14.2]{DR}; rather, we have to show that its proof applies to the points in $\Sigma$. This requires that the \( E \, \times \! \) CM curves are permuted by a Galois action, which is ensured by \cref{hecke-translate-cpts}.

The above arguments work for any Hecke--facteur family in a Shimura variety, thus proving \cref{main-thm-gen}.
Specialising to the \( E \, \times \! \) CM case,
the proof of \cref{galois-orbits-thm,galois-orbits-thm-sch} is obtained by adapting a height bound of Andr\'e for abelian varieties with large endomorphism rings \cite[Ch.~X, Theorem~1.3]{And89}, combined with the Masser--Wüstholz isogeny theorem.

\subsection{Outline of paper}

Section~\ref{sec:pre} contains various definitions, notation and basic facts relating to Shimura varieties, Hecke--facteur families and Siegel fundamental sets.
Section~\ref{sec:complexity} is largely devoted to comparing the two definitions of complexity used for \cref{galois-orbits-conj,galois-orbits-conj-sch} respectively.

In section~\ref{sec:galois-action}, we prove that a suitable Galois action permutes the \( E \, \times \! \) CM curves, as required to be able to apply the method of \cite[Theorem~14.2]{DR}.
Section~\ref{sec:congruence} proves a height bound for congruence subgroups which is another essential ingredient in the proof of \cref{main-thm}.
We carry out the strategy of \cite{DR} for Hecke--facteur families in Sections \ref{sec:param-height-bounds} and~\ref{sec:main-proof}.
Section \ref{sec:param-height-bounds} establishes height bounds for parameters of special subvarieties (modified versions of \cite[Conjectures 12.2 and~12.7]{DR}) while Section \ref{sec:main-proof} describes how we need to modify the proof of \cite[Theorem~14.2]{DR} to obtain \cref{main-thm,main-thm-gen}.

Finally Sections \ref{sec:andre} and~\ref{sec:galois-orbits} concern Galois orbits bounds for intersections with \( E \, \times \! \) CM curves.
Section~\ref{sec:andre} generalises the height bound of \cite[Ch.~X]{And89} to include abelian surfaces.
Section~\ref{sec:galois-orbits} applies this to prove \cref{galois-orbits-thm,galois-orbits-thm-sch}.

\subsection*{Acknowledgements}

The authors would like to thank Yves Andr\'e, David Holmes, David Masser and Andrei Yafaev for useful discussions during the preparation of this manuscript.
They would also like to thank the anonymous referees for helpful suggestions.
The first author is grateful to the University of Reading for financial support.
The second author thanks the EPSRC, who funded some of his work on this paper via grant EP/M020266/1, and the University of Warwick.

\section{Preliminaries} \label{sec:pre}

\subsection{Shimura varieties}

Let \( \bS \) denote the Deligne torus \( \Res_{\bC/\bR} \bG_m \).
A \defterm{Shimura datum} is a pair \( (\gG, X) \), where \( \gG \) is a connected reductive \( \bQ \)-algebraic group and \( X \) is a \( \gG(\bR) \)-conjugacy class in \( \Hom(\bS, \gG_\bR) \) satisfying \cite[axioms 2.1.1.1--2.1.1.3]{Del79}.
These axioms imply that \( X \) is a finite disjoint union of Hermitian symmetric domains \cite[Corollaire~1.1.17]{Del79}.

Given a Shimura datum \( (\gG, X) \) and a compact open subgroup \( K \subset \gG(\bA_f) \), the resulting \defterm{Shimura variety} is denoted \( \Sh_K(\gG, X) \).
This is a quasi-projective algebraic variety whose complex points are given by
\[ \Sh_K(\gG, X)(\bC) = \gG(\bQ) \bs X \times \gG(\bA_f) / K. \]
We can attach to each Shimura datum a number field called the \defterm{reflex field}, denoted \( E(\gG, X) \).
According to Deligne's theory, \( \Sh_K(\gG, X) \) has a so-called \defterm{canonical model} over \( E(\gG, X) \) (see \cite{Del79}, completed in \cite{MS82b}, \cite{Mil83} and \cite{Bor84}).

Given a morphism of Shimura data $(\gH,X_\gH)\rightarrow (\gG,X)$,
induced from a morphism $f:\gH\rightarrow\gG$ of $\QQ$-algebraic groups, and compact open subgroups $K_{\gH}\subset\gH(\AAA_f)$ and $K\subset\gG(\AAA_f)$ such that $f(K_\gH)\subset K$, we obtain a morphism of Shimura varieties
\[\Sh_{K_\gH}(\gH,X_\gH)\rightarrow\Sh_K(\gG,X),\]
which is closed and defined over the compositum of the reflex fields $E_\gH:=E(\gH,X_\gH)$ and $E_\gG:=E(\gG,X_\gG)$ \cite[Remark~13.8]{Mil05}. 
If $f$ is an inclusion, then $E_\gG\subset E_\gH$. 

Let \( \gH^\ad \) denote the quotient of \( \gH \) by its centre, and let \( f \colon \gH \to \gH^\ad \) be the natural morphism.
Then we obtain a Shimura datum \( (\gH^\ad, X^\ad_\gH) \) where $X^\ad_\gH$ is the $\gH^\ad(\RR)$-conjugacy class of morphisms containing the image of $X_\gH$ under composition with $f_\RR$.
The reflex fields are related by \( E(\gH^\ad, X^\ad_\gH) \subset E_\gH\).

A key example of a Shimura variety is \( \Ag \), the coarse moduli space of principally polarised abelian varieties of dimension~\( g \).
This Shimura variety is geometrically irreducible and defined over~\( \bQ \).
It is equal to \( \Sh_K(\gG, X) \) where \( \gG \) is the general symplectic group \( \gGSp_{2g} \), \( X \)~is isomorphic to the disjoint union of two copies of the Siegel upper half-space \( \Hg \) and \( K = \gGSp_{2g}(\hat\bZ) \).

For any algebraically closed field~\( k \) and any point \( s \in \Ag(k) \), we shall write \( A_s \) for the abelian variety parametrised by~\( s \) (which is defined up to \( k \)-isomorphism).

\subsection{Shimura variety components}

Over \( \bC \), a Shimura variety usually has many irreducible components.
It will often be convenient for us to work with a single geometrically irreducible component of a Shimura variety, which we call a \defterm{Shimura variety component}.

We define a \defterm{Shimura datum component} to be a pair \( (\gG, X^+) \), where \( (\gG, X) \) is a Shimura datum and \( X^+ \) is a connected component of \( X \).
Let \( \gG(\bQ)_+ \) denote the stabiliser of \( X^+ \) in \( \gG(\bQ) \).
The image of \( X^+ \) in \( \Sh_K(\gG, X) \) is a Shimura variety component, which we will denote by \( \Sh_K(\gG, X^+) \).
The complex points of \( \Sh_K(\gG, X^+) \) are isomorphic (as a complex analytic space) to \( \Gamma \bs X^+ \) where \( \Gamma = K \cap \gG(\bQ)_+ \) (subgroups of \( \gG(\bQ)_+ \) of this form are called \defterm{congruence subgroups}).

We say that \( (\gH, X_\gH^+) \) is a \defterm{Shimura subdatum component} of \( (\gG, X^+) \) if \( (\gG, X^+) \) and \( (\gH, X_\gH^+) \) are Shimura datum components, \( \gH \) is an algebraic subgroup of \( \gG \) and the inclusion \( \gH \hookrightarrow \gG \) induces an inclusion \( X_\gH^+ \hookrightarrow X^+ \).

Let \( (\gH, X^+_\gH) \) be a Shimura subdatum component of~\( (\gG, X^+) \).
Let \( \Gamma \subset \gG(\bQ)_+ \) be a congruence subgroup and let \( \Gamma_\gH = \Gamma \cap \gH(\bQ)_+ \).
Then \( X_\gH^+ \hookrightarrow X^+ \) induces a closed morphism \( \Gamma_\gH \bs X_\gH^+ \to \Gamma \bs X^+ \) between Shimura variety components.
We call the image of \( X_\gH^+ \to X^+ \) a \defterm{pre-special subvariety} of \( X^+ \) and the image of \( \Gamma_\gH \bs X_\gH^+ \to \Gamma \bs X^+ \) a \defterm{special subvariety} of \( \Gamma \bs X^+ \).

\subsection{Hecke correspondences}\label{sec:hecke-cor}

Let \( (\gG, X) \) be a Shimura datum and let \( K \subset \gG(\bA_f) \) be a compact open subgroup.
For each \( g \in \gG(\bA_f) \), there is an algebraic correspondence on \( \Sh_K(\gG, X) \) given by the morphisms
\[\Sh_K(\gG,X)\leftarrow\Sh_{K\cap gKg^{-1}}(\gG,X)\rightarrow\Sh_{g^{-1}Kg\cap K}(\gG,X)\rightarrow\Sh_K(\gG,X),\]
where the middle arrow is induced by the map 
\begin{align*}
(x,a)\rightarrow (x,ag) : X \times \gG(\bA_f) \to X \times \gG(\bA_f),
\end{align*}
and the outside arrows are the natural projections.
We call such a correspondence a \defterm{Hecke correspondence} and denote by $T_g$ the induced map on algebraic cycles.

\subsection{Hecke--facteur families} \label{sec:hecke-facteur}

Let \( (\gG, X^+) \) be a Shimura datum component and let \( (\gH, X_\gH^+) \subset (\gG, X^+) \) be a Shimura subdatum component.

Suppose that \( \gH \) has semisimple normal subgroups \( \gH_1 \) and \( \gH_2 \) such that \( \gH_1 \cap \gH_2 \) is finite and \( \gH^\der = \gH_1.\gH_2 \).
The groups \( \gH_1 \) and \( \gH_2 \) are not necessarily associated with Shimura data, but their adjoint groups give rise to Shimura datum components \( (\gH_1^\ad, X_1^+) \) and \( (\gH_2^\ad, X_2^+) \) such that there is a natural isomorphism \( X_\gH^+ \cong X_1^+ \times X_2^+ \).
We henceforth regard this isomorphism as an \emph{identification} and write \( X_1^+ \times X_2^+ = X_\gH^+ \subset X^+ \) without extra notation.
(In the language of \cite[Definition~4.4]{Mil05}, \( (\gH_1, X_1^+) \) and \( (\gH_2, X_2^+) \) are connected Shimura data.)

We call \( (\gH, X_\gH^+, \gH_1, \gH_2) \) as above a \defterm{facteur datum} in \( (\gG, X^+) \).

Let \( (\gH, X_\gH) \) denote the Shimura datum such that \( X_\gH^+ \) is a component of \( X_\gH \).
The decomposition \( X_\gH^+ \cong X_1^+ \times X_2^+ \) extends to a decomposition of Shimura data \( (\gH^\ad, X_\gH^\ad) \cong (\gH_1, X_1) \times (\gH_2, X_2) \).
The reflex field \( E(\gH^\ad, X_\gH^\ad) \) is equal to the compositum of \( E(\gH_1, X_1) \) and \( E(\gH_2, X_2) \).
If $K^\ad_\gH\subset\gH^\ad(\AAA_f)$ is a compact open subgroup equal to a product of compact open subgroups $K_1\subset\gH_1(\AAA_f)$ and $K_2\subset\gH_2(\AAA_f)$, then we obtain a decomposition of Shimura varieties
\[\Sh_{K^\ad_\gH}(\gH^\ad,X^\ad_\gH)=\Sh_{K_1}(\gH^\ad_1,X_1)\times\Sh_{K_2}(\gH^\ad_2,X_2),\]
which is defined over \( E(\gH^\ad, X_\gH^\ad) \).

For any pre-special point \( x_2 \in X_2^+ \), \( X_1^+ \times \{ x_2 \} \) is a pre-special subvariety of~\( X^+ \).
We call the collection of subvarieties of this form (for a fixed facteur datum) a \defterm{facteur family of pre-special subvarieties} of \( X^+ \).
We say that a facteur family is \defterm{trivial} if it comes from a facteur datum with \( \gH_2 = \{ 1 \} \): in this case, the facteur family consists simply of~\( X_\gH^+ \) itself.
This terminology is motivated by the notion of ``non-facteur'' special subvarieties from \cite{Ull07}:
a special subvariety is ``non-facteur'' if and only if it does not belong to any non-trivial facteur family.

The main topic of this paper will be the family of Hecke translates of a given facteur family.
For any \( g \in \gG(\bQ)_+ \) and any pre-special point \( x_2 \in X_2^+ \), 
\[ Y_{g,x_2} = g(X_1^+ \times \{ x_2 \}) \]
is again a pre-special subvariety of \( X^+ \).
We call the collection of subvarieties \( Y_{g,x_2} \) a \defterm{Hecke--facteur family of pre-special subvarieties} of \( X^+ \).
We call their images \( Z_{g,x_2} = \pi(Y_{g,x_2}) \), where $\pi$ denotes the uniformisation map $X^+\to \Gamma \bs X^+$, a  \defterm{Hecke--facteur family of special subvarieties} of \( \Gamma \bs X^+ \).
We chose this name because \( Z_{g,x_2} \) is an irreducible component of the image of \( \pi(X_1^+ \times \{ x_2 \}) \) (a special subvariety in the facteur family) under the Hecke correspondence \( T_{g^{-1}} \).

The Mumford--Tate group \( \MT(Y_{g,x_2})\) of a very general point in \( Y_{g,x_2} \) is equal to \( g(\gH_1.\gT_{x_2})g^{-1} \), where \( \gT_{x_2} \) is a torus in \( Z(\gH).\gH_2 \) corresponding to the pre-special point \( x_2 \in X_2^+ \).
Observe that
\( \MT(Y_{g,x_2})^\der = g \gH_1 g^{-1} \).

\subsection{Example of a Hecke--facteur family} \label{sec:hf-examples}

For a fundamental example of a Hecke--facteur family, choose positive integers \( d,e,g \) such that \( d+e = g \).
Let \( (\gG, X^+) = (\gGSp_{2g}, \Hg) \) and let \( X_\gH^+ \subset X^+ \) be a pre-special subvariety parametrising principally polarised abelian varieties which are isomorphic (as polarised abelian varieties) to a product \( A \times B \), where \( \dim(A) = d \), \( \dim(B) = e \) and \( A \), \( B \) are principally polarised.
The generic Mumford--Tate group of \( X_\gH^+ \) is
\[ \gH = \bG_m.(\gSp_{2d} \times \gSp_{2e}). \]
Let \( \gH_1 = \gSp_{2d} \subset \gH \) and \( \gH_2 = \gSp_{2e} \subset \gH \).
Then \( (\gH, X_\gH^+, \gH_1, \gH_2) \) is a facteur datum in \( (\gG, X^+) \), and $X^+_\gH=X^+_1\times X^+_2$, where $X^+_1=\mathcal{H}_d$ and $X^+_2=\mathcal{H}_e$.
The resulting Hecke--facteur family consists of the special subvarieties which parametrise principally polarised abelian varieties \( A \) of dimension~\( g \) such that there exist principally polarised abelian varieties \( A_1 \) (of dimension~\( d \)) and \( A_2 \) (of dimension~\( e \) and of CM type) and a polarised isogeny \( A_1 \times A_2 \to A \). (Given polarised abelian varieties \( (A, \lambda) \) and \( (B, \mu) \), a \defterm{polarised isogeny} is an isogeny \( f \colon A \to B \) such that \( f^* \mu = n\lambda \) for some \( n \in \bZ \).)

In this paper, we are principally interested in the case \( g=2 \), \( d=e=1 \) of this construction.
The following lemma shows that, in this case, the special curves in the Hecke--facteur family contain every point of \( \cA_2 \) for which the corresponding abelian variety is isogenous (not just \emph{polarised} isogenous) to a product \( E_1 \times E_2 \) where \( E_2 \) has CM.
In other words, this Hecke--facteur family consists precisely of the \( E \, \times \! \) CM curves defined in the introduction.

\begin{lemma} \label{excm-polarised}
Let \( (A, \lambda) \) be a principally polarised abelian surface for which there exist elliptic curves \( E_1 \) and \( E_2 \) (not in the same isogeny class) and an isogeny \( \varphi \colon E_1 \times E_2 \to A \).

Then there exist elliptic curves \( E_1' \) and \( E_2' \), such that \( E_i' \) is isogenous to \( E_i \), and a polarised isogeny \( \varphi' :E_1' \times E_2' \to (A, \lambda) \) (where \( E_1' \), \( E_2' \) are equipped with their principal polarisations).

Furthermore, \( \deg \varphi' \leq \deg \varphi \).
\end{lemma}

\begin{proof}
Let \( \Lambda = H_1(A, \bZ) \), equipped with the symplectic pairing \( \psi \) induced by the polarisation~\( \lambda \).
The isogeny \( \varphi \colon E_1 \times E_2 \to A \) induces an injection of \( \bZ \)-modules
\[ H_1(E_1, \bZ) \oplus H_1(E_2, \bZ) \to \Lambda \]
with finite cokernel.

Let \( \Lambda_i = (H_1(E_i, \bZ) \otimes_\bZ \bQ) \cap \Lambda \).
This is the \( \bZ \)-Hodge structure of an elliptic curve \( E_i' \) isogenous to \( E_i \).
The inclusion \( \Lambda_1 \oplus \Lambda_2 \to \Lambda \) induces an isogeny \( \varphi' \colon E_1' \times E_2' \to A \).

Let \( \lambda_i' \) be the (unique) principal polarisation on \( E_i' \) and let \( \psi_i' \) be the associated symplectic form on \( \Lambda_i \).
Because \( \rk \Lambda_i = 2 \), every symplectic form on \( \Lambda_i \) is an integer multiple of \( \psi_i' \).
(This is the step in the proof which is restricted to a product of elliptic curves.)
In particular \( \psi_{|\Lambda_i} = n_i \psi_i' \) for some \( n_i \in \bZ \).
Because \( \psi \) and \( \psi_i' \) can both be interpreted as the imaginary part of a positive definite Hermitian form (because they come from polarisations), we must have \( n_i > 0 \).

Because \( E_1 \) and \( E_2 \) are non-isogenous, there is no non-zero morphism of \( \bZ \)-Hodge structures \( \Lambda_1 \to \Lambda_2 \) or \( \Lambda_2 \to \Lambda_1 \).
Therefore \( \Lambda_1 \) and \( \Lambda_2 \) are orthogonal with respect to \( \psi \).
It follows that \( \psi = n_1 \psi'_1 + n_2 \psi'_2 \) and so \( \varphi'^* \lambda = (n_1 \lambda'_1, n_2 \lambda'_2) \).

By construction \( \Lambda_1 \) and \( \Lambda_2 \) are primitive submodules of \( \Lambda \).
Since \( \disc(n_i \psi_i') = n_i^2 \), \cref{symplectic-orthog-disc} implies that \( n_1 = n_2 \) so \( \varphi'^* \lambda \) is an integer multiple of \( (\lambda_1', \lambda_2') \), that is, \( \varphi' \) is a polarised isogeny.

Furthermore,
\[ H_1(E_1, \bZ) \oplus H_1(E_2, \bZ) \subset \Lambda_1 \oplus \Lambda_2 \subset \Lambda. \]
It follows that \( \varphi \) factors through \( \varphi' \) and so \( \deg \varphi' \leq \deg \varphi \).
\end{proof}

In the proof of \cref{excm-polarised}, we needed the following lemma.
This is a symplectic version of \cite[Chapter~14, Proposition~0.2]{Huy16}, which is the analogous result for symmetric bilinear forms.

\begin{lemma} \label{symplectic-orthog-disc}
Let \( \Lambda \) be a free \( \bZ \)-module of finite rank with a perfect symplectic form \( \psi \colon \Lambda \times \Lambda \to \bZ \).
Let \( \Lambda_1, \Lambda_2 \) be primitive submodules of \( \Lambda \) which are orthogonal to each other and such that \( \Lambda_1 \oplus \Lambda_2 \) has finite index in \( \Lambda \).
Then
\[ \disc(\psi_{|\Lambda_1}) = \disc(\psi_{|\Lambda_2}) = [\Lambda : \Lambda_1 + \Lambda_2]. \]
\end{lemma}

\begin{proof}
Let
\[ \Lambda_i^\vee = \{ v \in \Lambda_i \otimes \bQ : \psi(v, \Lambda_i) \subset \bZ \}. \]
The map \( v \mapsto \psi_{|\Lambda_i}(v, -) \) is an isomorphism \( \Lambda_i^\vee \to \Hom(\Lambda_i, \bZ) \).
Consequently, there is a unique \( \bZ \)-module homomorphism \( \alpha_i \colon \Lambda \to \Lambda_i^\vee \) such that
\[ \psi_{|\Lambda_i}(v, -) = \psi_{|\Lambda_i}(\alpha_i(v), -) \]
for all \( v \in \Lambda \).

We shall show that \( \alpha_i \) induces an isomorphism \( \Lambda / (\Lambda_1 + \Lambda_2) \to \Lambda_i^\vee / \Lambda_i \) for each~\( i \).
This will suffice to prove the lemma because \( \disc(\psi_{|\Lambda_i}) = [\Lambda_i^\vee : \Lambda_i] \).

Firstly, \( \alpha_i \colon \Lambda \to \Lambda_i^\vee \) is surjective because \( \Lambda_i \) is a primitive submodule of~\( \Lambda \)
 and \( \psi \) is a perfect pairing on \( \Lambda \).

For every \( v \in \Lambda \), \( v - \alpha_1(v) \) is orthogonal to \( \Lambda_1 \).
Thus \( v - \alpha_1(v) \in \Lambda_2 \otimes \bQ \).

If \( \alpha_1(v) \in \Lambda_1 \), then \( v - \alpha_1(v) \in \Lambda \).
Since \( v - \alpha_1(v) \in \Lambda_2 \otimes \bQ \) and \( \Lambda_2 \) is a primitive submodule of \( \Lambda \), we deduce that \( v - \alpha_1(v) \in \Lambda_2 \) and hence \( v \in \Lambda_1 + \Lambda_2 \).
Conversely, if \( v \in \Lambda_1 + \Lambda_2 \), then writing \( v = v_1 + v_2 \) with \( v_i \in \Lambda_i \), we get \( \alpha_1(v) = v_1 \in \Lambda_1 \).

Thus \( \alpha_1(v) \in \Lambda_1 \) if and only if \( v \in \Lambda_1 + \Lambda_2 \).
In other words, \( \alpha_1 \) induces an injection \( \Lambda / (\Lambda_1 + \Lambda_2) \to \Lambda_1^\vee / \Lambda_1 \).
A similar argument applies to \( \alpha_2 \).
\end{proof}

\subsection{Special subvarieties of \texorpdfstring{\( \cA_2 \)}{A2}} \label{sec:a2-subvars}

The main theorems of this paper concern special subvarieties of \( \cA_2 \).
As a consequence of the classification of Hodge groups of abelian surfaces in \cite{MZ99}, every special subvariety of \( \cA_2 \) is an irreducible component of the locus of principally polarised abelian surfaces whose endomorphism algebras contain a fixed algebra (``a Shimura variety of PEL type'').

\begin{table}[t]
\caption{Special subvarieties of \( \cA_2 \)} \label{tab:a2-special}
\begin{tabular}{lcp{15.5em}}
   Description
 & Dimension
 & Generic endomorphism algebra
\\[0.4em] \hline \rule{0pt}{1.2em}%
   \( \cA_2 \)
 & 3
 & \( \bQ \)
\\[0.4em]
   Hilbert modular surface
 & 2
 & real quadratic field
\\[0.4em]
   Hecke translate of \( \cA_1 \times \cA_1 \)
 & 2
 & \( \bQ \times \bQ \)
\\[0.4em]
   Shimura curve
 & 1
 & indefinite quaternion algebra over~\( \bQ \)
\\[0.4em]
   modular curve
 & 1
 & \( \rM_2(\bQ) \)
\\[0.4em]
   \( E \, \times \! \) CM curve
 & 1
 & \( \bQ \times F \), where \( F \) is an imaginary quadratic field
\\[0.4em]
   simple CM point
 & 0
 & quartic CM field
\\[0.4em]
   non-simple CM point
 & 0
 & product of two distinct imaginary quadratic fields
\\[0.4em]
   isotypic CM point
 & 0
 & \( \rM_2(F) \), where \( F \) is an imaginary quadratic field
\\[0.4em] \hline
 &&
\end{tabular}
\end{table}

In \cref{tab:a2-special}, we list the classes of special subvarieties of \( \cA_2 \), giving the dimension of each special subvariety and the generic endomorphism algebra of the abelian surfaces parametrised by that special subvariety.

Unlikely intersections with special points are taken care of by the André--Oort conjecture (proved for \( \cA_2 \) by Pila and Tsimerman \cite{PT13}).
Therefore, in order to prove the Zilber--Pink conjecture for \( \cA_2 \), it remains only to consider intersections between special curves and a general curve.
In this paper, we consider intersections with the \( E \, \times \! \) CM special curves. The outstanding special curves will be treated in a forthcoming article by the same authors \cite{DO20}.

\subsection{Fundamental sets}

Let \( (\gG, X^+) \) be a Shimura datum component and let \( (\gH, X_\gH^+, \gH_1, \gH_2) \) be a facteur datum.
Let \( (\gH_1^\ad, X_1^+) \) and \( (\gH_2^\ad, X_2^+) \) denote the Shimura datum components such that \( X_\gH^+ \cong X_1^+ \times X_2^+ \).
We now describe how to choose compatible fundamental sets in \( X_1^+ \), \( X_2^+ \) and \( X^+ \).

\begin{definition}
A \defterm{Siegel fundamental set} in \( X^+ \) for a congruence subgroup \( \Gamma \subset \gG(\bQ)_+ \) is a fundamental set for \( \Gamma \) of the form \( C.\fS^+.x_0 \), where
\( C \subset \gG(\bQ)_+ \) is a finite set, \( \fS^+ = \fS \cap \gG(\bR)^+ \) for some Siegel set \( \fS \subset \gG(\bR) \) and \( x_0 \in X^+ \) is a point such that the stabiliser of \( x_0 \) in \( \gG(\bR) \) right-stabilises \( \fS \).
We use the definitions of Siegel sets and associated terminology from \cite[section~2B]{Orr18}.
\end{definition}

\begin{lemma} \label{fundamental-sets}
Fix a congruence subgroup \( \Gamma \subset \gG(\bQ)_+ \) and let \( \Gamma_1 = \Gamma \cap \gH_1(\bQ)_+ \), \( \Gamma_2 = \Gamma \cap \gH_2(\bQ)_+ \).
Let \( \cF_1 \subset X_1^+ \) and \( \cF_2 \subset X_2^+ \) be Siegel fundamental sets for the congruence subgroups \( \Gamma_1 \) and \( \Gamma_2 \), respectively.

Then there exists a Siegel fundamental set \( \cF \subset X^+ \) for \( \Gamma \) such that \( \cF_1 \times \cF_2 \subset \cF \).
\end{lemma}

The proof of \cref{fundamental-sets} relies on the following lemma.

\begin{lemma} \label{siegel-set-product}
Let \( \fS_1 \subset \gH_1(\bR) \) and \( \fS_2 \subset \gH_2(\bR) \) be Siegel sets.
Then \( \fS_1.\fS_2 \) is a Siegel set in \( \gH^\der(\bR) \).
\end{lemma}

\begin{proof}
We use the notation from \cite[section~2B]{Orr18}, adding subscripts \( 1 \), \( 2 \) or \( \gH^\der \) as appropriate. For example, \( (\gP_1, \gS_1, K_1) \) denotes the Siegel triple associated with the Siegel set \( \fS_1 \).

We begin by constructing a Siegel triple for \( \gH^\der \).
Multiplication in \( \gH \) is a central \( \bQ \)-isogeny \( \gH_1 \times \gH_2 \to \gH^\der \), so \( \gP_{\gH^\der} = \gP_1.\gP_2 \) is a minimal parabolic \( \bQ \)-subgroup of \( \gH^\der \).
Similarly \( K_{\gH^\der} = K_1.K_2 \) is a maximal compact subgroup of \( \gH^\der(\bR) \).

Let \( \gS_{\gH^\der} = \gS_1.\gS_2 \).
This is an \( \bR \)-torus in \( \gP_{\gH^\der} \).
Since \( \gS_i \) is \( \gP_i(\bR) \)-conjugate to a maximal \( \bQ \)-split torus in \( \gH_i \), we can use the central \( \bQ \)-isogeny to deduce that \( \gS_{\gH^\der} \) is \( \gP_{\gH^\der}(\bR) \)-conjugate to a maximal \( \bQ \)-split torus in \( \gP_{\gH^\der} \). Finally \( \gS_{\gH^\der} \) is stabilised by the Cartan involution of \( \gH^\der \) associated with \( K_{\gH^\der} \), because this Cartan involution restricts to the Cartan involutions of \( \gH_1 \) and \( \gH_2 \) associated with \( K_1 \) and \( K_2 \) respectively.
Thus \( (\gP_{\gH^\der}, \gS_{\gH^\der}, K_{\gH^\der})\) is a Siegel triple for \( \gH^\der \).

The unipotent radical of \( \gP_{\gH^\der} \) is \( \gU_{\gH^\der} = \gU_1.\gU_2 \).
The isogeny \( \gH_1 \times \gH_2 \to \gH^\der \) induces central isogenies \( Z_{\gH_1}(\gS_1) \times Z_{\gH_2}(\gS_2) \to Z_{\gH^\der}(\gS_{\gH^\der}) \) and \( \gP_1/\gU_1 \times \gP_2/\gU_2 \to \gP_{\gH^\der}/\gU_{\gH^\der} \), with the latter being defined over \( \bQ \).
Hence the maximal \( \bQ \)-anisotropic subgroup of \( \gP_{\gH^\der}/\gU_{\gH^\der} \) is the product of the corresponding subgroups in \( \gP_1/\gU_1 \) and \( \gP_2/\gU_2 \).
Lifting to \( Z_{\gH^\der}(\gS_{\gH^\der}) \), it follows that \( \gM_{\gH^\der} = \gM_1.\gM_2 \).

The set of simple roots of \( \gH^\der \) (with respect to \( (\gP_{\gH^\der}, \gS_{\gH^\der}) \)) is the union of the sets of simple roots of \( \gH_1 \) and \( \gH_2 \).
It follows that \( A_{\gH^\der,t} = A_{1,t}.A_{2,t} \).

We have \( \fS_i = \Omega_i.A_{i,t}.K_i \) for \( i = 1 \) or \( 2 \), where \( \Omega_i \) is a compact subset of \( \gU_i(\bR).\gM_i(\bR)^+ \).
Since \( \gM_1 \) commutes with \( \gU_2 \), \( \Omega_1.\Omega_2 \) is a compact subset of \( \gU_{\gH^\der}(\bR).\gM_{\gH^\der}(\bR)^+ \).
Hence
\[ \fS_{\gH^\der} = \Omega_1.\Omega_2.A_{1,t}.A_{2,t}.K_1.K_2 \]
is a Siegel set in \( \gH^\der(\bR) \).
Because \( \gH_1 \) commutes with \( \gH_2 \), we conclude that \( \fS_{\gH^\der} = \fS_1.\fS_2 \).
\end{proof}

\begin{proof}[Proof of \cref{fundamental-sets}]
Write \( \cF_i = C_i.\fS_i^+.x_i \) as in the definition of Siegel fundamental sets.

By \cref{siegel-set-product}, \( \fS_1.\fS_2 \) is a Siegel set in \( \gH^\der(\bR) \).
By \cite[Theorem~4.1]{Orr18}, there exists a Siegel set \( \fS \subset \gG(\bR) \) and a finite set \( C \subset \gG(\bQ) \) such that \( \fS_1.\fS_2 \subset C.\fS \).
Let \( C_+ = C \cap \gG(\bQ)_+ \) and \( \fS^+ = \fS \cap \gG(\bR)^+ \).
Enlarging \( C \) and \( \fS \) if necessary, we may ensure that \( C_+.\fS^+ \) is a fundamental set for \( \Gamma \) in \( \gG(\bR)_+ \).

Let \( x_0 \) be the image of \( (x_1, x_2) \) under the inclusion \( X_1^+ \times X_2^+ \to X^+ \).
Let \( K_i = \Stab_{\gH_i(\bR)}(x_i) \).
Examining the proof of \cite[Theorem~4.1]{Orr18} and in particular \cite[Lemma~4.4]{Orr18}, we see that the maximal compact subgroup \( K_\gG \) in the Siegel triple used to construct \( \fS \) can be chosen to be any maximal compact subgroup of \( \gG(\bR) \) satisfying two conditions: \( K_\gG \) contains \( K_1.K_2 \) and the Cartan involution of \( \gG \) associated with \( K_\gG \) stabilises \( \gH^\der \).
In fact, if we let \( K_\gG \) be the maximal compact subgroup of \( \gG(\bR) \) which stabilises \( x_0 \), then it satisfies these conditions.
Indeed, \( K_1.K_2 \subset K_\gG \) since \( K_1.K_2 = \Stab_{\gH^\der(\bR)}(x_0) \).
The Cartan involution of~\( \gG \) associated with \( K_\gG \) is conjugation by \( x_0(i) \).
Since \( x_0(i) \in \gH(\bR) \), this Cartan involution stabilises \( \gH^\der(\bR) \) as required.

It follows that \( \cF = C_+.\fS^+.x_0 \) is a fundamental set for \( \Gamma \) in \( X^+ \) which contains \( \fS_1^+.x_1 \times \fS_2^+.x_2 \).
Replacing \( C \) by \( C_1.C_2.C \), we get \( \cF_1 \times \cF_2 \subset \cF \).
\end{proof}

\subsection{Heights and determinants}\label{sec:heights}


Let $k\geq 1$ be an integer. For any real number~$y$, we define its \defterm{$k$-height} as
\begin{align*}
{\rm H}_k(y):=\min\{\max\{|a_0|,...,|a_k|\}:a_i \in \ZZ,\ {\rm gcd}\{a_0,...,a_k\}=1,\ a_ky^k+...+a_0=0\},
\end{align*}
where we use the convention that, if the set is empty, that is, $y$ is not algebraic of degree less than or equal to~$k$, then ${\rm H}_k(y)$ is $+\infty$. For $y=(y_1,...,y_m)\in \RR^m$, we set
\begin{align*}
{\rm H}_k(y):=\max\{{\rm H}_k(y_1),...,{\rm H}_k(y_m)\}.
\end{align*}
We extend this definition to $\CC^m$ by identifying it with $\RR^{2m}$, taking real and imaginary parts.
The $1$-height of a matrix $g\in\rM_n(\bQ)$ is the height of $g$ considered as an element of $\QQ^{n^2}$. 
For any matrix \( g \in \rM_n(\bQ) \), we write
\[ \detstar(g) = \abs{\det(g)} \cdot (\lcm \{ b_{ij} : 1 \leq i,j \leq n \})^n \in \bZ, \]
where the entries of \( g \) written in lowest terms are \( a_{ij}/b_{ij} \).

\section{Complexities} \label{sec:complexity}

We define a notion of complexity \( \Delta(Z) \) for special subvarieties in a Hecke--facteur family which is similar to the general definition of complexity of special subvarieties from \cite{DR}, but modified to be more convenient for the case of Hecke--facteur families.
We will then define a second notion of complexity \( \Delta'(Z) \) (which is even more specialised to the \( E \, \times \! \) CM case) and show that \( \Delta \) and \( \Delta' \) are polynomially bounded in terms of each other.

Let \( (\gG, X^+) \) be a Shimura datum component, let \( K \subset \gG(\bA_f) \) be a compact open subgroup and let \( S  \) be the Shimura variety component \( \Sh_K(\gG, X^+) \).
Choose a faithful representation \( \gG \to \gGL_{m,\bQ} \) (we shall use this representation to talk about the height or \( \detstar \) of elements of \( \gG(\bQ) \)).

For each special point \( s \in S \), one can define the following objects and quantities associated with \( s \) (as in \cite[Definition~10.1]{DR}):
\begin{enumerate}
\item \( \gT \subset \gG \) is the Mumford--Tate group of \( s \).
\item \( K_\gT^m \) is the maximal compact open subgroup of \( \gT(\bA_f) \). (There is a unique maximal compact open subgroup because $\gT$ is a torus.)
\item \( D_\gT \) is the absolute value of the discriminant of the splitting field of \( \gT \).
\item \( \Delta(s) = \max\{D_\gT, \, [K_\gT^m : K \cap \gT(\bA_f)]\} \).
\end{enumerate}
The \( \bQ \)-torus \( \gT \) is defined up to conjugation by \( K \cap \gG(\bQ)_+ \),
and \( D_\gT \) and \( \Delta(s) \) are independent of the choice of \( \gT \) in its conjugacy class.

Let \( (\gH, X_\gH^+, \gH_1, \gH_2) \) be a facteur datum for \( (\gG, X^+) \).
For any special subvariety \( Z \subset S \) in the associated Hecke--facteur family, we define the following quantities:
\begin{enumerate}
\item \( N(Z) \) is the smallest positive integer~\( N \) such that there exist \( \gamma \in \gG(\bQ)_+ \) and \( x_2 \in X_2^+ \) with \( Z = Z_{\gamma,x_2} \) and \( \detstar(\gamma) = N \).
\item \( \Delta(Z) = \max\{N(Z), \, \min \{ \Delta(s) : s \in Z \text{ is a special point} \}\} \).
\end{enumerate}
We call \( \Delta(Z) \) the \defterm{complexity} of \( Z \).

In the examples from section~\ref{sec:hf-examples}, the following lemma shows that the quantity \( N(Z) \) can be interpreted as the smallest positive integer \( N \) such that, for every point \( s \in Z \), there exist abelian varieties \( A_1 \) and \( A_2 \) (of dimensions \( d \) and \( e \), respectively) and a polarised isogeny \( A_1 \times A_2 \to A_s \) of degree~\( N \). 

\begin{lemma} \label{detstar-isog-comparison}
Let \( d \), \( e \) and \( g \) be positive integers such that \( d+e=g \) and let \( (\gH, X_\gH^+, \gH_1, \gH_2) \) be the facteur datum described in section~\ref{sec:hf-examples}.
Let \( Z \) be a special subvariety of \( \Ag \) in the Hecke--facteur family associated with \( (\gH, X_\gH^+, \gH_1, \gH_2) \).

For each positive integer~\( N \), the following statements are equivalent:
\begin{enumerate}[(i)]
\item There exists \( \gamma \in \gGSp_{2g}(\bQ)_+ \) and a pre-special point \( x_2 \in X_2^+ \) such that \( Z = Z_{\gamma,x_2} \) and \( \detstar(\gamma) = N \)
(with respect to the inclusion \( \gGSp_{2g} \hookrightarrow \gGL_{2g} \)).
\item There exists a pre-special point \( x_2 \in X_2^+ \) such that, for every point \( s \in Z \), there exists an abelian variety \( A_1 \) of dimension~\( d \) and a polarised isogeny \( A_1 \times A_{x_2} \to A_s \) of degree~\( N \), where $A_{x_2}$ denotes the abelian variety of dimension $e$ naturally associated with $x_2\in X^+_2=\mathcal{H}_e$.
\item There exists a pre-special point \( x_2 \in X_2^+ \), a point $s \in Z$ such that $Z$ is the smallest special subvariety containing $s$, an abelian variety \( A_1 \) of dimension~\( d \) and a polarised isogeny \( A_1 \times A_{x_2} \to A_s \) of degree~\( N \).
\end{enumerate}
\end{lemma}

\begin{proof}
First suppose that we are given \( \gamma \) and $x_2$ as in (i).
We can multiply \( \gamma \) by the lowest common multiple of the denominators of its entries without changing \( \detstar(\gamma) \), so we may assume that \( \gamma \in \gGSp_{2g}(\bQ)_+ \cap \rM_{2g}(\bZ) \).
By the definition of \( Z_{\gamma,x_2} \), for each point \( s \in Z \), we have \( s = \pi(\gamma.(x_1, x_2)) \) for some point \( x_1 \in X_1^+ \).
Then the matrix~\( \gamma \) is the rational representation of a polarised isogeny \( A_{x_1} \times A_{x_2} \to A_s \) of degree \( \det(\gamma) = N \), where $A_{x_1}$ denotes the abelian variety of dimension $d$ naturally associated with $x_1\in X^+_1=\mathcal{H}_d$. This yields (ii). 

It is obvious that (ii) implies (iii).

Finally, suppose that (iii) holds.
Let \( p_1, p_2 \) denote the first and second projections \( \Hg \times \Hg \to \Hg \) and
let \( \pi \) denote the uniformising map \( \Hg \to \Ag \).
Let
\[ T_N = \{ (s_1, s_2) \in \Ag \times \Ag : \exists \text{ a polarised isogeny } A_{s_1} \to A_{s_2} \text{ of degree } N \} \]
and let \( X_N = (\pi\times\pi)^{-1}(T_N) \).
Let 
\[ W_N = \bigl( (X_1^+ \times \{x_2\}) \times \Hg \bigr) \cap X_N. \]
Each irreducible component of \( X_N \) is of the form \( X_\gamma = \{ (x, \gamma x) : x \in \Hg \} \) for some \( \gamma \in \gGSp_{2g}(\bQ)_+ \cap \rM_{2g}(\bZ) \) such that \( \det(\gamma) = N \).
Hence, each irreducible component of \( W_N \) has the form
\[ W_\gamma = \{ ((x_1, x_2), \gamma(x_1, x_2)) : x_1 \in \cH_d \} \] for some \( \gamma \) of the same type.

By (iii), \( s \in (\pi \circ p_2)(W_N) \).
Each irreducible component of \( (\pi \circ p_2)(W_N) \) is a special subvariety of \( \cA_g \).
Since \( Z \) is the smallest special subvariety containing~\( s \), we deduce that \( Z \subset (\pi \circ p_2)(W_N) \).
Since \( Z \) is irreducible, it is contained in \( (\pi \circ p_2)(W_\gamma) \) for some~\( \gamma \).
Because \( Z \) is in the Hecke--facteur family, \( \dim(Z) = \dim(W_\gamma) \).
Therefore, \( Z = (\pi \circ p_2)(W_\gamma) = Z_{\gamma,x_2} \).
\end{proof}

Now suppose that \( (\gH, X_\gH^+, \gH_1, \gH_2) \) is the facteur datum discussed in section~\ref{sec:hf-examples} with $g=2$ and $d=e=1$, that is, the datum associated with \( E \, \times \! \) CM curves.
In this case, we can relate the above definition of complexity to a simpler definition, which is crucial for the proof of \cref{galois-orbits-thm-sch}.
The new definition of complexity relies on the following lemma.

\begin{lemma} \label{min-isogeny-unique}
Let \( (A, \lambda) \) be a principally polarised abelian surface which is isogenous to a product of elliptic curves \( E_1 \times E_2 \) where \( E_1 \) and \( E_2 \) are not isogenous to each other.
Let \( N \) be the smallest positive integer such that there exist elliptic curves \( E_1 \) and \( E_2 \) and an isogeny \( E_1 \times E_2 \to A \) of degree~\( N \).

Then the pair of elliptic curves \( (E_1, E_2) \) such that there exists an isogeny \( E_1 \times E_2 \to A \) of degree~\( N \) is unique up to swapping \( E_1 \) and \( E_2 \).
\end{lemma}

\begin{proof}
Let \( E_1 \) and \( E_2 \) be elliptic curves for which there exists an isogeny \( \varphi \colon E_1 \times E_2 \to A \) of degree~\( N \).
Let \( \Lambda = H_1(A, \bZ) \) and \( \Lambda_i = H_1(E_i, \bZ) \).
Let \( \psi \) be the symplectic form on \( \Lambda \) induced by \( \lambda \) and let \( \psi_i \) be the symplectic form on \( \Lambda_i \) induced by the unique principal polarisation of \( E_i \).

Because \( \deg \varphi \) is as small as possible, we must have \( \varphi' = \varphi \) in \cref{excm-polarised}.
Looking at the proof of \cref{excm-polarised}, we see that there is a positive integer~\( n \) such that \( \psi_{|\Lambda_i} = n\psi_i \).
Then \( N = \deg \varphi = n^2 \).

\medskip

As a first step, we show that \( \ker(\varphi) \subset (E_1 \times E_2)[n] \).

Choose a basis \( \{ x, y \} \) for \( \Lambda_1 \) such that \( \psi_1(x, y) = 1 \).
For any \( v \in \Lambda \), we can write \( v = v_1 + v_2 \) where \( v_i \in \Lambda_i \otimes \bQ \).
We can calculate
\[ \psi(v, x) = \psi(v_1, x) = n\psi_1(v_1, x). \]
Since \( \psi(v, x) \in \bZ \), \( \psi_1(v_1, x) \in \frac{1}{n}\bZ \).
Similarly \( \psi_1(v_1, y) \in \frac{1}{n}\bZ \).
Now
\[ v_1 = \psi_1(v_1,y).x - \psi_1(v_1, x).y \]
and so \( v_1 \in \frac{1}{n}\Lambda_1 \).

A similar argument shows that \( v_2 \in \frac{1}{n}\Lambda_2 \).
Thus \( \Lambda \subset \frac{1}{n} \Lambda_1 + \frac{1}{n} \Lambda_2 \).
In other words, \( n \) annihilates \( \Lambda / (\Lambda_1 + \Lambda_2) \cong \ker(\varphi) \).

\medskip

Suppose that there exists another pair of elliptic curves \( E_1' \), \( E_2' \) and an isogeny \( \varphi' \colon E_1' \times E_2'\to A \) also of degree~\( N \).

We can apply the above argument to \( \varphi' \), deducing that \( \ker(\varphi') \subset (E_1' \times E_2')[n] \).
Consequently, there exists an isogeny \( \psi' \colon A \to E_1' \times E_2' \) such that \( \psi' \circ \varphi' = [n]_{E_1' \times E_2'} \).
We will consider the isogeny \( \psi' \circ \varphi \colon E_1 \times E_2 \to E_1' \times E_2' \).

Label \( E_1' \) and \( E_2' \) so that they are isogenous to \( E_1 \) and \( E_2 \) in that order.
Because \( E_1 \) is not isogenous to \( E_2 \), we have \( \Hom(E_1, E_2') = \Hom(E_2, E_1') = 0 \).
Consequently
\[ \psi' \circ \varphi = (\alpha_1, \alpha_2) \]
for some isogenies \( \alpha_1 \colon E_1 \to E_1' \) and \( \alpha_2 \colon E_2 \to E_2' \).

Let \( G = \ker(\varphi) \subset (E_1 \times E_2)[n] \) and let \( p_i \) denote the projection \( (E_1 \times E_2)[n] \to E_i[n] \) for \( i = 1 \) and \( 2 \).

Observe that \( \varphi \colon E_1 \times E_2 \to A \) factors through \( (E_1 \times E_2) / (G \cap E_2[n]) \) because $G \cap E_2[n]\subset G = \ker(\varphi)$.
Hence the minimality of \( \deg \varphi \) implies that \( G \cap E_2[n] = 0 \).
We conclude that \( p_{1|G} \) is injective.
It follows that \( \# p_1(G) = \# G = n^2 \).
Thus \( p_1(G) = E_1[n] \).
Similarly \( p_2(G) = E_2[n] \).

If \( g = (g_1, g_2) \in G \), then \( (\alpha_1(g_1), \alpha_2(g_2)) = \psi'\varphi(g) = 0 \) and so \( p_i(G) \subset \ker(\alpha_i) \).

Since \( \deg(\varphi') = n^2 \) and \( \deg [n]_{E_1' \times E_2'} = n^4 \), we have \( \deg(\psi') = n^2 \).
Consequently \( \deg(\psi' \circ \varphi) = n^4 \).
Meanwhile \( \ker(\psi' \circ \varphi) = \ker(\alpha_1) \times \ker(\alpha_2) \) so the inclusions \( p_i(G) \subset \ker(\alpha_i) \) must be equalities.

It follows that \( E_i' \cong E_i/E_i[n] \cong E_i \).
\end{proof}

We continue to consider the \( E \, \times \! \) CM Hecke--facteur family, using the standard representation $\gGSp_4\to\gGL_4$ to define \( \detstar \).
Let $Z$ be a special curve in this family. By Lemma \ref{detstar-isog-comparison}, there exists $x_2\in X^+_2$ such that, for every $s\in Z$, there exists an elliptic curve $E_1(s)$ and an isogeny $E_1(s)\times E_{x_2}\rightarrow A_s$ of degree $N(Z)$, where $E_{x_2}$ is the elliptic curve naturally associated with $x_2\in X^+_2=\mathcal{H}_1$. Furthermore, for any point $s\in Z\cap\Sigma$, $N(Z)$ is the smallest degree of an isogeny of $A_s$ to a product of elliptic curves. Therefore, by \cref{min-isogeny-unique}, $E_1(s)$ and $E_{x_2}$ are the unique elliptic curves whose product has an isogeny to $A_s$ of degree $N(Z)$.
Since \( E_{x_2} \) has CM while \( E_1(s) \), \( E_{x_2} \) is uniquely associated with \( Z \).
Hence, it makes sense to define
\[ \Delta'(Z) = \max\{N(Z), \, \abs{\disc(\End(E_{x_2}))} \}. \]

\begin{remark} \label{delta's}
If $s$ is an \( E \, \times \! \) CM point in the base of a principally polarised abelian scheme \( \fA \to V \), as defined in section~\ref{sec:alt-state}, then the point of $\cA_2$ attached to the fiber $\fA_s$ lies in a unique \( E \, \times \! \) CM curve $Z$.
In this situation (as for example in the statement of \cref{galois-orbits-conj-sch}), we write $N(s) = N(Z)$ and $\Delta'(s) = \Delta'(Z)$.
\end{remark}

For a special curve \( Z \) in the \( E \, \times \! \) CM Hecke--facteur family, our two notions of complexity \( \Delta(Z) \) and \( \Delta'(Z) \) are polynomially bounded with respect to each other.
Before proving this, we shall prove two lemmas.
The first is a special case of the reverse of \cite[Lemma~7.2]{Tsi12} or \cite[Theorem~4.1]{DO16}.
(We expect that the reverse of \cite[Theorem~4.1]{DO16} should hold for all special points in an arbitrary Shimura variety, but the proof is likely to be somewhat more complicated.)

\begin{lemma} \label{delta-disc-bound}
There exist absolute constants \( \newC{delta-disc-mult} \) and \( \newC{delta-disc-exp} \) such that for every special point \( s \in \cA_2 \), if there exists a polarised isogeny \( E_1 \times E_2 \to A_s \) where \( E_1 \) and \( E_2 \) are non-isogenous elliptic curves both of which have CM, then
\[ \Delta(s) \leq \refC{delta-disc-mult} \, \abs{\disc(\End(A_s))}^{\refC{delta-disc-exp}}. \]
\end{lemma}

\begin{proof}
Let \( F_i = \End(E_i) \otimes \bQ \) for \( i = 1 \) and \( 2 \).
Let \( \cO_i \) denote the ring of integers of \( F_i \) and let \( \cO = \cO_1 \times \cO_2 \), \( \cO_p = \cO \otimes \bZ_p \) and \( \hat\cO = \prod_p \cO_p \).
Let \( R = \End(A_s) \), \( R_p = R \otimes \bZ_p \) and \( \hat{R} = \prod_p R_p \).
Observe that \( R \) is an order in \( F_1 \times F_2 \) so
\[ \disc(R) = \disc(F_1) \disc(F_2) [\cO:R]^2. \]

We use the notation \( \gT \), \( K_\gT^m \) and \( D_\gT \) as in the definition of \( \Delta(s) \) in section~\ref{sec:complexity}.

Fix a basis for \( H_1(A_s, \bZ) \).
With respect to this basis, we get a homomorphism of rings \( \iota \colon R \to \rM_4(\bZ) \)
and an injection \( \gT \to \gGL_{4,\bQ} \).

Because the Mumford--Tate group of an abelian surface is always as large as possible given its endomorphism ring and polarisation (see section~\ref{sec:a2-subvars}), \( \gT \) is equal to the intersection of \( \gGSp_4 \) with the centraliser (in \( \gGL_4 \)) of \( \iota(R) \).
Now \( \iota(R \otimes \bQ) \) is a commutative algebra of dimension~\( 4 \), so it is its own centraliser in \( \rM_4(\bQ) \).
Hence
\[ \gT(\bQ) = \iota((R \otimes \bQ)^\times) \cap \gGSp_4(\bQ). \]
We deduce that \( \gT \) is a subtorus of \( \Res_{F_1/\bQ} \bG_m \times \Res_{F_2/\bQ} \bG_m \) and hence is split over the compositum \( F_1 F_2 \).
Consequently
\[ D_\gT \leq \abs{\disc(F_1 F_2)} \leq \abs{\disc(F_1)}^2 \abs{\disc(F_2)}^2 \leq \abs{\disc(R)}^2 \]
(the middle inequality is well-known; see for example \cite[Exercise~8.10]{Jar14}).

Let
\[ K_\gT = \gT(\bA_f) \cap \gGSp_4(\hat\bZ) = \gT(\bA_f) \cap \iota(\hat{R}^\times). \]
The maximal compact subgroup of \( ((F_1 \times F_2) \otimes \bA_f)^\times \) is \( \hat\cO^\times \), so \( K_\gT^m = \gT(\bA_f) \cap \iota(\hat\cO^\times) \).
Therefore
\[ [K_\gT^m:K_\gT] = [\gT(\bA_f) \cap \iota(\hat\cO^\times) : \gT(\bA_f) \cap \iota(\hat{R}^\times)]
   \leq [\hat\cO^\times : \hat{R}^\times]. \]
Thus it will suffice to show that
\[ [\hat\cO^\times:\hat{R}^\times] \leq [\cO:R]^4. \]
We will prove this prime by prime: for each prime~\( p \) we will show that
\begin{equation} \label{unit-index-p}
[\cO_p^\times:R_p^\times] \leq [\cO_p:R_p]^4.
\end{equation}

Let \( N_p = [\cO_p:R_p] \).
If \( \cO_p = R_p \), then \eqref{unit-index-p} is obvious.
So we may assume that \( \cO_p \neq R_p \) and then \( p \) divides \( N_p \).

Since \( 1 \in R_p \) and \( N_p\cO_p \subset R_p \), we have \( 1 + N_p\cO_p \subset R_p \).
Because \( p \) divides \( N_p \), if \( x \in 1 + N_p\cO_p \) then \( x \in \cO_p^\times \) and \( x^{-1} \) is also congruent to \( 1 \) mod \( N_p \).
In other words \( x^{-1} \in 1 + N_p \cO_p \subset R_p \) and so \( x \in R_p^\times \).
This proves that \( 1 + N_p\cO_p \subset R_p^\times \).
Hence
\[ [\cO_p^\times:R_p^\times] \leq [\cO_p^\times : 1 + N_p\cO_p]. \]
Finally the inclusion \( \cO_p^\times \to \cO_p \) induces an injective map of sets \( \cO_p^\times / (1 + N_p\cO_p) \to \cO_p/N_p\cO_p \)
and so
\[ [\cO_p^\times : 1 + N_p\cO_p] \leq [\cO_p:N_p\cO_p] = N_p^4. \]
This completes the proof.
\end{proof}

Our second lemma is a general result on endomorphisms and isogenies of abelian varieties.
Write \( Z(R) \) for the centre of a ring \( R \).
(Note that the analogous statement to \cref{isogeny-discriminants} still holds if you replace \( Z(\End(A)) \) and \( Z(\End(B)) \) by \( \End(A) \) and \( \End(B) \) respectively, with the same modification to the proof.)

\begin{lemma} \label{isogeny-discriminants}
Let \( A \) and \( B \) be abelian varieties such that that there exists an isogeny \( \varphi \colon A \to B \) of degree~\( N \), and let \( r = \rk_\bZ(Z(\End(A))) \).
Then
\[ \abs{\disc(Z(\End(B)))} \leq N^{2r} \abs{\disc(Z(\End(A)))}. \]
\end{lemma}

\begin{proof}
Equip \( Z(\End(A)) \otimes \bQ \) and \( Z(\End(B)) \otimes \bQ \) with the symmetric bilinear forms \( (x,y) \mapsto \tr(xy) \), where \( \tr \) denotes the trace on the respective \( \bQ \)-algebras.
We denote these forms \( \psi_A \) and \( \psi_B \).
Recall that by definition,  \( \disc(Z(\End(A))) \) is the discriminant of the bilinear form \( \psi_A \) on \( Z(\End(A)) \), and similarly for \( B \).

The isogeny \( \varphi \) induces an isomorphism \( \iota \colon \End(A) \otimes \bQ \to \End(B) \otimes \bQ \) given by \( \iota(\alpha) = \varphi\alpha\varphi^{-1} \).
This restricts to an isomorphism \( Z(\End(A)) \otimes \bQ \to Z(\End(B)) \otimes \bQ \).
Any isomorphism of \( \bQ \)-algebras preserves the trace, and therefore is compatible with the bilinear forms \( \psi_A \) and \( \psi_B \).

Because there exists an isogeny \( \psi \colon B \to A \) such that \( \psi \circ \varphi = [N] \), \( \iota \) maps \( N \cdot \End(A) \) into \( \End(B) \), and hence \( N \cdot Z(\End(A)) \) into \( Z(\End(B)) \).
Consequently
\[ \abs{\disc(\psi_B|Z(\End(B)))}  \leq  \abs{\disc(\psi_A | N \cdot Z(\End(A)))}  =  N^{2r} \abs{\disc(\psi_A | Z(\End(A)))}.
\qedhere
\]
\end{proof}

Now we combine the above two lemmas to compare \( \Delta \) and \( \Delta' \).

\begin{lemma}\label{comp-poly-equiv}
There exist positive constants \( \newC{complex-const-1} \), \( \newC{complex-power-1} \), \( \newC{complex-const-2} \), and \( \newC{complex-power-2} \) such that, if $Z$ is a special curve in the \( E \, \times \! \) CM Hecke-facteur family, then
\[\refC{complex-const-1}\Delta(Z)^{\refC{complex-power-1}}\leq\Delta'(Z)\leq \refC{complex-const-2}\Delta(Z)^{\refC{complex-power-2}}.\]
\end{lemma}

\begin{proof}
Write \( Z = Z_{\gamma,x_2} \) where \( \gamma \in \gGSp_4(\bQ)_+ \cap \rM_4(\bZ) \), \( x_2 \in X_2^+ \) and \( \detstar(\gamma) = N(Z) \).

Fix two CM elliptic curves \( E_1 \) and \( E_1' \) which are not isogenous to each other.
Assume that \( E_{x_2} \) is not isogenous to \( E_1 \) (if \( E_{x_2} \) is isogenous to \( E_1 \), then use \( E_1' \) instead of \( E_1 \) and apply the same argument).

Let \( x_1 \) be a point in \( X_1^+ \) corresponding to the elliptic curve \( E_1 \), and let \( s \) denote the image of \( \gamma(x_1, x_2) \) in \( \cA_2 \).
Then \( s \) is a special point contained in \( Z \), so
\[ \Delta(Z) \leq \max \{ N(Z), \Delta(s) \}. \]
By \cref{delta-disc-bound}, we have
\[ \Delta(s) \leq \refC{delta-disc-mult} \, \abs{\disc(\End(A_s))}^{\refC{delta-disc-exp}}. \]
Now \( \gamma \) is the rational representation of a polarised isogeny \( E_1 \times E_{x_2} \to A_s \) of degree \( \detstar(\gamma) = N(Z) \).
Noting that \( \rk_\bZ(\End(E_1 \times E_{x_2})) = 4 \) and that \( \End(E_1 \times E_{x_2}) \) and \( \End(A_s) \) are commutative, \cref{isogeny-discriminants} tells us that
\[ \abs{\disc(\End(A_s))}  \leq  N(Z)^8 \abs{\disc(\End(E_1 \times E_{x_2}))}. \]
Since \( E_{x_2} \) is not isogenous to \( E_1 \), \( \End(E_1 \times E_{x_2}) = \End(E_1) \times \End(E_{x_2}) \) and so
\[ \abs{\disc(\End(E_1 \times E_{x_2}))} = \abs{\disc(\End(E_1)) \disc(\End(E_{x_2}))}
   \leq \refC{max-disc-e1} \Delta'(Z) \]
where \( \newC{max-disc-e1} = \max\{\abs{\disc(\End(E_1))}, \abs{\disc(\End(E_1'))} \} \).

Combining the above inequalities gives
\[ \refC{complex-const-1}\Delta(Z)^{\refC{complex-power-1}}\leq\Delta'(Z). \]

To prove the second inequality of the lemma, let $s\in Z$ be a special point such that $\Delta(s)$ is minimal.
By \cref{detstar-isog-comparison}, there exists an elliptic curve $E_1(s)$ and an isogeny $E_1(s)\times E_{x_2} \to A_s$ of degree $N(Z)$.
Since both \( E_1(s) \times E_{x_2} \) and \( A_s \) are principally polarised, we deduce that there exists an isogeny in the reverse direction \( A_s \to E_1(s)\times E_{x_2} \), also of degree \( N(Z) \).
Hence \cref{isogeny-discriminants} implies that
\[ \abs{\disc(Z(\End(E_1(s) \times E_{x_2})))} \leq N(Z)^8 \abs{\disc(Z(\End(A_s)))}. \]

Let \( \cO_s \) denote the maximal order in \( Z(\End(A_s)) \otimes \bQ \) (which is either an imaginary quadratic field or a product of two imaginary quadratic fields).
Then
\[ \disc(Z(\End(A_s))) = [\cO_s : Z(\End(A_s))]^2 \disc(Z(\End(A_s)) \otimes \bQ). \]
By \cite[Lemma~7.2]{Tsi12}, we have (using the notation from the definition of \( \Delta(s) \))
\[ [\cO_s : Z(\End(A_s))] \leq [K_\gT^m : \gGSp_{2g}(\hat\bZ) \cap \gT(\bA_f)]^{\newC*} D_\gT^{\newC*} \leq \Delta(s)^{\newC*}. \]
Combining the above inequalities, we conclude that
\begin{equation} \label{eqn:z-inequality}
\abs{\disc(Z(\End(E_1(s) \times E_{x_2})))} \leq \Delta(s)^{\newC*} \abs{\disc(Z(\End(A_s)) \otimes \bQ)}.
\end{equation}

We now split into two cases depending on whether \( E_1(s) \) is isogenous to \( E_{x_2} \) or not.
If they are not isogenous, then \( F_1(s) = \End(E_1(s)) \otimes \bQ \) and \( F_2 = \End(E_{x_2}) \otimes \bQ \) are distinct imaginary quadratic fields and we have
\[ Z(\End(E_1(s) \times E_{x_2})) = \End(E_1(s)) \times \End(E_{x_2}) \]
so that
\[ \abs{\disc(\End(E_{x_2}))} \leq \abs{\disc(Z(\End(E_1(s) \times E_{x_2})))} \]
while
\[ \abs{\disc(Z(\End(A_s)) \otimes \bQ)} = \abs{\disc(F_1(s)) \disc(F_2)} \leq \abs{\disc(F_1(s) F_2)}^2 = D_\gT^2. \]
Combining these with \eqref{eqn:z-inequality} completes the proof of the second inequality in this case.

If \( E_1(s) \) is isogenous to \( E_{x_2} \), then \( \End(E_1(s) \times E_{x_2}) \otimes \bQ \cong \rM_2(F) \) where \( F = \End(E_1(s)) \otimes \bQ \cong \End(E_{x_2}) \otimes \bQ \).
Hence \( Z(\End(E_1(s) \times E_{x_2})) \) is an order in \( F \).
If \( \iota \colon E_{x_2} \to E_1(s) \times E_{x_2} \) and \( \beta \colon E_1(s) \times E_{x_2} \to E_{x_2} \) denote the inclusion and projection morphisms, then \( \alpha \mapsto \beta\alpha\iota \) is an injection \( Z(\End(E_1(s) \times E_{x_2})) \to \End(E_{x_2}) \).
Hence
\[ \abs{\disc(\End(E_{x_2}))} \leq \abs{\disc(Z(\End(E_1(s) \times E_{x_2})))}. \]
The splitting field of \( \gT \) is \( F \cong Z(\End(A_s) \otimes \bQ \), so
\[ \abs{\disc(Z(\End(A_s)) \otimes \bQ)} = D_\gT. \]
Again combining these with \eqref{eqn:z-inequality} completes the proof.
\end{proof}

\section{Galois action on Hecke--facteur special subvarieties} \label{sec:galois-action}

Let \( S \) be a Shimura variety component defined over a number field \( E_S \).
Then \( \Gal(\Qbar/E_S) \) acts on the set of \( \Qbar \)-subvarieties of \( S \).
We shall show that, after restricting to \( \Gal(\Qbar/F) \) for a suitable field extension \( F/E_S \), this action permutes the special subvarieties in a given Hecke--facteur family.

The main result of this section (\cref{galoisonhecke}) shares a certain similarity with \cite[Conjecture~12.6]{DR}.
Indeed, \cite[Conjecture~12.6]{DR} predicts that, for any special subvariety \( Z \) in a suitable facteur family, we can make an extension of the base field (whose degree is very small relative to the complexity of \( Z \)) over which all Galois conjugates of \( Z \) are members of finitely many \emph{facteur} families.
\Cref{galoisonhecke} is weaker than this because it only asserts that the Galois conjugates are members of the same \emph{Hecke--facteur} family. On the other hand, \cref{galoisonhecke} only involves an extension of the base field which is independent of~\( Z \). 

Given a Shimura datum component $(\gG,X^+)$, let $X$ denote the $\gG(\RR)$-conjugacy class of morphisms $\bS \to \gG(\bR)$ containing $X^+$. Then $(\gG,X)$ is a Shimura datum.
For any subset $A$ of $X$ and $a\in\gG(\AAA_f)$, we will denote by $[A,a]_K$ the image of $A\times\{a\}$ in $\Sh_K(\gG,X)(\CC)$, in analogy with a standard notation for points. In particular, the geometrically irreducible components of $\Sh_K(\gG,X)$ are the subsets of the form $[X^+,a]_K$ for any $a\in\gG(\AAA_f)$.

We begin with a lemma on the restriction of Hecke correspondences on $\Sh_K(\gG, X)$ to the connected component $[X^+, 1]_K$.
Recall that \( T_g \) denotes the map on algebraic cycles induced by the Hecke correspondence associated with~\( g \) (see section \ref{sec:hecke-cor}).

\pagebreak 

\begin{lemma} \label{hecke-translate-cpts}
Let \( (\gG, X^+) \) be a Shimura datum component and let $K\subset\gG(\AAA_f)$ be a compact open subgroup. 
Let \( S\) denote the Shimura variety component $[X^+ ,1]_K \subset \Sh_K(\gG,X)$.

Let \( Z \) be an irreducible complex algebraic subvariety of \( S \) and let \( Y \) be an irreducible component of \( \pi^{-1}(Z) \).
Let \( g \in \gG(\bQ)_+ \).
Then every irreducible component of \( T_{g^{-1}}(Z) \cap S \) can be written in the form \( [g'Y ,1]_K \) for some \( g' \in \gG(\bQ)_+ \).
\end{lemma}

\begin{proof}
Each irreducible component of \( T_{g^{-1}}(Z) \) has the form
\[ Z' = [Y,kg^{-1}]_K \]
for some \( k \in K \).
If \( Z' \subset S \), then \( kg^{-1} \in \gG(\bQ)_+K \), or in other words
\[ kg^{-1}k' \in \gG(\bQ)_+ \]
for some \( k' \in K \).
Letting \( g' = (kg^{-1}k')^{-1} \), we get
\[ Z' = [Y, g'^{-1}]_K = [g'Y, 1]_K. \qedhere\]
\end{proof}

\begin{remark}\label{rem-det}
Fix a faithful representation $\gG\to\gGL_{m,\bQ}$ such that the image of $K$ is contained in $\gGL_m(\hat{\ZZ})$.
Since \( g' = k'^{-1} g k^{-1} \), we have \( \abs{\det(g')} = \abs{\det(g)} \).
We also claim that the lowest common multiples of the denominators of \( g' \) and \( g \) are equal.
This is because \( g' \) is obtained from \( g \) by multiplying on the left and right by elements of \( K\).
Working prime by prime, we see that such operations cannot increase the lowest common multiple of the denominators.
Since also \( g = k'g'k \), these operations cannot decrease the lowest common multiple of the denominators either.
We conclude that \( \detstar(g') = \detstar(g) \).
\end{remark}

\begin{proposition}\label{galoisonhecke}
Let \( (\gG, X^+) \) be a Shimura datum component and let $K\subset\gG(\AAA_f)$ be a compact open subgroup.
Let \( S\) denote the Shimura variety component $[X^+,1]_K \subset \Sh_K(\gG,X)$.

Let \( (\gH, X_\gH^+, \gH_1, \gH_2) \) be a facteur datum for \( (\gG, X^+) \) and let \( \fF \) denote the associated Hecke--facteur family of special subvarieties of \( S \).

Let $E_S$ denote the extension of \( E_\gG:=E(\gG, X) \) over which $S$ is defined.
Let $F$ denote the compositum of $E_\gH:=E(\gH,X_\gH)$ and $E_S$.

For every special subvariety \( Z \in \fF \), and every \( \sigma \in \Gal(\Qbar / F) \), the Galois conjugate subvariety \( \sigma(Z) \) is again a member of the Hecke--facteur family \( \fF \).
In other words, \( \sigma(Z) = Z_{g_\sigma,x_\sigma} \) for some \( g_\sigma \in \gG(\bQ)_+ \) and some pre-special point \( x_\sigma \in X_2^+ \).
\end{proposition}

\begin{proof}
First, let $g\in\gG(\QQ)_+$ and $x_2\in X^+_2$ such that $Z=Z_{g,x_2}$. Now let $Z_{x_2}:=[X^+_1,x_2]_{K_\gH}\subseteq\Sh_{K_\gH}(\gH,X_\gH)$, where $K_\gH:=\gH(\AAA_f)\cap K$, and let
\begin{align*}
\varphi_\gH:\Sh_{K_\gH}(\gH,X_\gH)\rightarrow\Sh_K(\gG,X)
\end{align*}
denote the natural morphism of Shimura varieties.
By definition, $Z_{g,x_2}$ is an irreducible component of $T_{g^{-1}}(\varphi_\gH(Z_{x_2}))$.
Since $T_{g^{-1}}$ is defined over $E_\gG\subseteq F$ and $\varphi_\gH$ is defined over $E_\gH\subseteq F$, we see that, for any $\sigma\in\Gal(\Qbar / F)$, the subvariety $\sigma(Z_{g,x_2})$ is an irreducible component of $T_{g^{-1}}(\varphi_\gH(\sigma(Z_{x_2})))$.

Let $K^\ad_\gH\subset\gH^\ad(\AAA_f)$ denote a maximal compact open subgroup containing the image of $K_{\gH}$ under the natural morphism $q \colon \gH(\AAA_f) \rightarrow \gH^\ad(\AAA_f)$. 
Then $K^\ad_\gH=K_1\times K_2$ for maximal compact subgroups $K_1\subset\gH^\ad_1(\AAA_f)$ and $K_2\subset\gH^\ad_2(\AAA_f)$ and we obtain a finite morphism of Shimura varieties
\[\varphi^\ad:\Sh_{K_{\gH}}(\gH,X_\gH)\rightarrow\Sh_{K^\ad_\gH}(\gH^\ad,X^\ad_\gH)=\Sh_{K_1}(\gH^\ad_1,X_1)\times\Sh_{K_2}(\gH^\ad_2,X_2),\]
which is defined over $E_\gH$.
We conclude that, for any \( \sigma \in \Gal(\Qbar / F) \),
\begin{align*}
\varphi^\ad(\sigma(Z_{x_2}))=\sigma(\varphi^\ad(Z_{x_2}))
&=\sigma([X^+_1,1]_{K_1})\times\sigma([\{x_2\},1]_{K_2}).
\end{align*}
The first factor is a geometrically connected component of $\Sh_{K_1}(\gH^\ad_1,X_1)$ because $\Gal(\Qbar / F)$ acts on the set of such components.
The second factor is a special point of $\Sh_{K_2}(\gH^\ad_2,X_2)$ by \cite[Theorem~5.1]{OrrSV}.
Hence
\begin{align*}
\varphi^\ad(\sigma(Z_{x_2}))=\sigma(\varphi^\ad(Z_{x_2}))
&=[X^+_1,a_\sigma]_{K_1}\times[\{y_\sigma\},b_\sigma]_{K_2},
\end{align*}
for some pre-special point $y_\sigma\in X^+_2$ and some $a_\sigma\in\gH^\ad_1(\AAA_f)$ and $b_\sigma\in\gH^\ad_2(\AAA_f)$.
In other words, $\sigma(Z_{x_2})$ is an irreducible component of 
\[(\varphi^\ad)^{-1}([X^+_1,a_\sigma]_{K_1}\times[\{y_\sigma\},b_\sigma]_{K_2}),\]
and so $\sigma(Z_{x_2})$ is equal to $[X^+_1\times\{x_\sigma\},h_\sigma]_{K_\gH}$ for some pre-special point $x_\sigma\in X^+_2$ and some element $h_\sigma\in\gH(\AAA_f)$ such that $q(h_\sigma) = (a_\sigma,b_\sigma)$.

On the other hand, because $E_S \subseteq F$, we have \[\varphi_\gH(\sigma(Z_{x_2}))=\sigma(\varphi_\gH(Z_{x_2}))\subseteq S,\] 
for any $\sigma\in\Gal(\Qbar/F)$.
Hence $[X^+,h_\sigma]_K = S$ and so $h_\sigma=h_{\sigma,\QQ}k$ for some $h_{\sigma,\QQ}\in\gG(\QQ)_+$ and $k\in K$. In other words,
\[\varphi_\gH(\sigma(Z_{x_2}))=[X^+_1\times\{x_\sigma\},h_{\sigma,\QQ}]_K=[h_{\sigma,\QQ}^{-1}(X^+_1\times\{x_\sigma\}),1]_K.\]

Since $\sigma(Z_{g,x_2})$ is an irreducible component of $T_{g^{-1}}(\varphi_\gH(\sigma(Z_{x_2})))$, by \cref{hecke-translate-cpts},
\[ \sigma(Z_{g,x_2}) = [g'h_{\sigma,\bQ}^{-1}(X_1^+ \times \{x_\sigma\}),1]_K \]
for some \( g' \in \gG(\bQ)_+ \).
Thus, letting \( g_\sigma = g'h_{\sigma,\bQ}^{-1} \), we can write \( Z = Z_{g_\sigma,x_\sigma} \).
\end{proof}

\begin{remark}
Using Remark \ref{rem-det}, it is possible to augment the proof of \cref{galoisonhecke} to show that, given a faithful representation $\gG\to\gGL_{m,\bQ}$, the complexities of $Z$ and $\sigma(Z)$ differ by at most a constant multiple independent of $Z$. However, the argument is more technical and also requires \cite[Theorem~1.3]{OrrSV}.
\end{remark}

\section{Height bound for representatives modulo congruence subgroups} \label{sec:congruence}

We will need the following proposition on heights of certain elements in arithmetic groups.
The \( \gG = \gSL_2 \) case of this proposition can be proved by the procedure outlined in \cite[Exercise~1.2.2]{DS05}.
An effective version of the proposition for \( \gG = \gSO_Q \) (where \( Q \) is an integral quadratic form) can be found at \cite[Theorem~8]{LM16}.

\begin{proposition} \label{lift-bound}
Let \( \gG \) be a semisimple \( \bQ \)-algebraic group and let \( \Gamma \subset \gG(\bQ) \) be a congruence subgroup.
Let \( \rho \colon \gG \to \gGL_{m,\bQ} \) be a faithful representation.
For each positive integer \( n \), let
\[ \Gamma(n) = \Gamma \cap \rho^{-1}(\ker(\gGL_m(\bZ) \to \gGL_m(\bZ/n\bZ))). \]
Then there exist constants \( \newC{lift-mult} \) and \( \newC{lift-exp} \) such that, for each positive integer~\( n \), every class in \( \Gamma / \Gamma(n) \) has a representative in \( \Gamma \) whose $1$-height is at most \( \refC{lift-mult} n^{\refC{lift-exp}} \).
\end{proposition}

Our proof of \cref{lift-bound} relies on the notion of expander families of graphs.
An \defterm{expander family} is an infinite family of finite graphs (in which self-loops and multiple edges are permitted) for which there exists \( \epsilon > 0 \) such that for every graph~\( G \) in the family,
\[ \min \{ \# \partial X/\# X : X \subset V(G), 0 < \# X \leq \# V(G) /2 \} \geq \epsilon, \]
and there is a uniform upper bound for the degrees of the vertices of all graphs in the family.
Here \( V(G) \) means the set of vertices of~\( G \) and \( \partial X \) denotes the set of edges of \( G \) which have one endpoint in \( X \) and the other endpoint not in \( X \).

The relevance of expander families comes from the following theorem.

\begin{theorem} \label{expanders}
Let \( \gG \) be a semisimple \( \bQ \)-algebraic group.
Let \( \rho \colon \gG \to \gGL_{m,\bQ} \) be a faithful representation and let \( \Gamma = \rho^{-1}(\gGL_m(\bZ)) \).
For each positive integer \( n \), let \( \pi_n \colon \Gamma \to \gGL_m(\bZ/n\bZ) \) denote the composition of \( \rho \) with reduction modulo~\( n \).

Suppose that \( \Gamma \) is infinite (equivalently, \( \gG \) is of non-compact type).
Let \( \Delta \) be a finite generating set for \( \Gamma \) such that \( \Delta^{-1} = \Delta \).

The Cayley graphs \( \Cay(\pi_n(\Gamma), \pi_n(\Delta)) \) form an expander family as \( n \) ranges over the positive integers.
\end{theorem}

\begin{proof}
Let \( \tilde\gG \) denote the simply connected cover of \( \gG \).
If \( \tilde\gG \) is simple of real rank at least~\( 2 \), then it has Kazhdan's property~(\(T\)) \cite{Vas68}.
In general, Clozel \cite[Theorem~3.1]{Clo03} (using a result of Burger and Sarnak \cite{BS91}) proved that the trivial representation is isolated in the support of the regular representation of \( \tilde\gG(\bR) \) on \( L^2(\tilde\gG(\bQ) \bs \tilde\gG(\bA)) \).
Thanks to \cite[Proposition~4.7]{LZ03}, this implies that \( \Gamma \) has property~(\(\tau\)) with respect to its congruence subgroups.
By \cite[Proposition~1.2]{LZ89}, this implies the theorem.
%
\end{proof}

We will now use \cref{expanders} to prove \cref{lift-bound}.

\begin{proof}[Proof of \cref{lift-bound}]
The proposition is obvious if \( \Gamma \) is finite.
Hence we may assume that \( \Gamma \) is infinite.

We begin by reducing to the case where \( \Gamma = \rho^{-1}(\gGL_m(\bZ)) \).
Since \( \Gamma \) is a congruence subgroup, there exists \( n_0 \) such that
\[ \rho^{-1}(\ker(\gGL_m(\bZ) \to \gGL_m(\bZ/n_0\bZ)) = \Gamma(n_0) \subset \Gamma. \]
Since \( \Gamma(n_0n) \subset \Gamma(n) \), it suffices to prove the lemma with \( \Gamma(n_0n) \) in place of \( \Gamma(n) \) (for any \( n \in \bZ \)).
Assuming that the proposition holds for \( \rho^{-1}(\gGL_m(\bZ)) \), we get a set of representatives for \( \rho^{-1}(\gGL_m(\bZ)) / \Gamma(n_0 n) \) of height bounded in terms of \( n \).
A subset of these representatives then form a set of representatives for \( \Gamma(n_0) / \Gamma(n_0 n) \).
Since \( \Gamma(n_0) \) has finite index in \( \Gamma \), we can choose once and for all a set of representatives \( \gamma_1, \dotsc, \gamma_s \) for \( \Gamma / \Gamma(n_0) \).
Then we can obtain representatives for every class in \( \Gamma/\Gamma(n) \) by multiplying the \( \gamma_i \) by representatives for \( \Gamma(n_0) / \Gamma(n_0 n) \).

Thus it suffices to assume that \( \Gamma = \rho^{-1}(\gGL_m(\bZ)) \).

According to \cite[Theorem~6.5]{BHC62}, \( \Gamma \) is finitely generated, say by a set~\( \Delta \).
Enlarging \( \Delta \), we may assume that \( \Delta^{-1} = \Delta \).
Hence we can apply \cref{expanders} to show that the Cayley graphs \( \Cay_n = \Cay(\pi_n(\Gamma), \pi_n(\Delta)) \) form an expander family.

By \cite[Corollary~3.1.10]{Kow}, the graphs in the expander family \( \{ \Cay_n \} \) satisfy
\[ \diam(\Cay_n)  \ll  \log(3\,\#V(\Cay_n)). \]
Since \( \# V(\Cay_n) = \# \pi_n(\Gamma) \leq n^{m^2} \) we deduce that
\[ \diam(\Cay_n)  \leq  \refC{diameter-mult} \log(3n) \]
for some constant \( \newC{diameter-mult} > 0 \).

By the definition of the Cayley graph, any element \( g \in \pi_n(\Gamma) \) can be written as the product of at most \( \diam(\Cay_n) \) elements of \( \pi_n(\Delta) \).
Consequently we can find some \( \tilde{g} \in \Gamma \) such that \( \pi_n(\tilde{g}) = g \) and \( \tilde{g} \) is a product of at most \( \diam(\Cay_n) \) elements of \( \Delta \).
Because \( \rH_1(xy) \leq m\rH_1(x)\rH_1(y) \) for all matrices \( x \) and \( y \), we deduce that
\[ \rH_1(\tilde{g})  \leq  m^{\diam(\Cay_n)-1} \, \max\{\rH_1(\delta):\delta\in\Delta\}^{\diam(\Cay_n)} \]
and thus
\[ \rH_1(\tilde{g})  \leq \newC*^{\refC{diameter-mult} \log(3n)} = \newC* \, n^{\newC*} \]
for suitable constants.
\end{proof}

\begin{remark}
The semisimplicity condition is necessary in \cref{lift-bound}: the proposition fails for the group \( \gG = \Res_{F/\bQ} \bG_m \), where \( F \) is a real quadratic field.
\end{remark}

\section{Parameter height bounds for Hecke--facteur families} \label{sec:param-height-bounds}

The aim of this section is to prove versions of \cite[Conjectures 12.2 and~12.7]{DR} for Hecke--facteur families, modified to use \( \detstar(g) \) in place of \( \deg(\pi(Y_{g,x_2})) \) (the same modification as we made to the definition of complexity).

\begin{lemma} \label{hf-fund-desc}
Let \( (\gG, X^+) \) be a Shimura datum component and let \( S = \Gamma \bs X^+ \) be an associated Shimura variety component. Fix a faithful representation $\rho:\gG\to\gGL_{m,\bQ}$.
Let \( (\gH, X_\gH^+, \gH_1, \gH_2) \) be a facteur datum for \( (\gG, X^+) \) and choose fundamental sets \( \cF_1 \subset X_1^+ \), \( \cF_2 \subset X_2^+ \) and \( \cF \subset X^+ \) as in \cref{fundamental-sets}.

Then there exist constants \( \newC{hf-fund-desc-mult} \) and \( \newC{hf-fund-desc-exp} \) such that, if \( Y \) is a pre-special subvariety in the Hecke--facteur family associated with \( (\gH, X_\gH^+, \gH_1, \gH_2) \) satisfying \( Y \cap \cF \neq \emptyset \),
then there exist \( g \in \gG(\bQ)_+ \) and \( x_2 \in \cF_2 \) satisfying \( Y = Y_{g,x_2} \), \( \detstar(g) = N(\pi(Y)) \) and
\[ \rH_1(g)  \leq  \refC{hf-fund-desc-mult} \, N(\pi(Y))^{\refC{hf-fund-desc-exp}}. \]
\end{lemma}

\begin{proof}
By the definition of the Hecke--facteur family,
we can write
\[ Y = Y_{g',x_2'}  = g'(X^+_1\times\{x_2'\}) \]
for some \( g' \in \gG(\bQ)_+ \) and \( x_2' \in X_2^+ \).
We may choose \( g' \) so that \( \detstar(g') = N(\pi(Y)) \).

Pick a point \( x \in Y \cap \cF \) (which we are assuming to be non-empty).
Then $g'^{-1}x \in X^+$ so by \cref{fundamental-sets}, we can choose
\[ \gamma \in (\Gamma_1 \cap \gH_1(\RR)_+).(\Gamma_2 \cap \gH_2(\RR)_+) \subset \Gamma \cap \gH(\RR)_+ \]
such that \( \gamma^{-1} g'^{-1} x \in \cF_1 \times \cF_2 \subset \cF\).

Write \( g = g'\gamma \) and let \( g^{-1}x = (x_1, x_2) \).
Here \( x_i \in \cF_i \).
Because \( \gH_1 \) commutes with \( \gH_2 \), we have \( Y = Y_{g,x_2} \).

Because \( x \) and \( g^{-1}x \) are both in \( \cF \), \cite[Theorem~1.1]{Orr18} gives a polynomial bound for \( \rH_1(g) \) in terms of
\[ \abs{\det(g)} \cdot (\max\{b_{ij} : 1 \leq i,j \leq n\})^n \]
where the entries of \( g \) are written in lowest terms as \( a_{ij}/b_{ij} \) (with \( b_{ij} > 0 \)).
Since \( \max\{b_{ij}\} \leq \lcm\{b_{ij}\} \), this implies that \( H_1(g) \) is polynomially bounded in terms of
\( \detstar(g) = \detstar(g') = N(\pi(Y)) \).
\end{proof}

The following lemma is a modified special case of \cite[Conjecture~12.7]{DR}.

\begin{proposition} \label{conj127-det}
Let \( (\gG, X^+) \), \( S \), $\rho$, \( (\gH, X_\gH^+, \gH_1, \gH_2) \), \( \cF \), \( \cF_1 \) and \( \cF_2 \) be as in \cref{hf-fund-desc}. 

Then there exist constants \( \newC{conj127-det-mult} \) and \( \newC{conj127-det-exp} \) such that,
if \( Y \) is a pre-special subvariety in the Hecke--facteur family associated with \( (\gH, X_\gH^+, \gH_1, \gH_2) \) satisfying \( Y \cap \cF \neq \emptyset \) and \( z \in \cF \cap \Gamma Y \), then there exists \( \gamma \in \Gamma \) satisfying \( z \in \gamma.Y \) and
\[ \rH_1(\gamma)  \leq  \refC{conj127-det-mult} \, N(\pi(Y))^{\refC{conj127-det-exp}}. \]
\end{proposition}

\begin{proof}
Write \( Y = Y_{g,x_2} \) as in \cref{hf-fund-desc}.

We are given \( z \in \cF \cap \Gamma Y \), so we can pick \( \gamma' \in \Gamma \) such that \( z \in \gamma'Y \).
Then
\[ g^{-1} \gamma'^{-1}z \in X_1^+ \times \{ x_2 \}. \]
Because \( \cF_1 \) is a fundamental set in \( X_1^+ \), we can choose \( \gamma_1 \in \Gamma_1 \) such that
\[ \gamma_1^{-1} g^{-1} \gamma'^{-1}z \in \cF_1 \times \{ x_2 \}. \]

Because we chose fundamental sets as in \cref{fundamental-sets}, \( \cF_1 \times \{ x_2 \} \subset \cF \).
Since also \( z \in \cF \), by \cite[Theorem~1.1]{Orr18} we get that
\( \rH_1(\gamma' g \gamma_1) \) is polynomially bounded in terms of \( \detstar(\gamma' g \gamma_1) = \detstar(g) = N(\pi(Y)) \).
Because \( \gamma', \gamma_1 \in \Gamma \), \( \detstar(\gamma' g \gamma_1) = \detstar(g) = N(\pi(Y)) \).
Thus there are constants \( \newC{ggg-mult} \) and~\( \newC{ggg-exp} \) such that
\begin{equation} \label{eqn:ggg}
\rH_1(\gamma' g \gamma_1)  \leq  \refC{ggg-mult} N(\pi(Y))^{\refC{ggg-exp}}.
\end{equation}

We can choose \(n \), polynomially bounded in terms of \( \rH_1(g) \), such that \( \Gamma \cap g^{-1} \Gamma g \) contains the principal congruence subgroup \( \Gamma(n) \) (defined with respect to $\rho$).
Then \( \Gamma(n) \cap \Gamma_1 \) is a principal congruence subgroup in \( \gH_1(\bQ) \) (with respect to $\rho_{|\gH_1}$).
Applying \cref{lift-bound} to \( \gH_1 \), we can choose \( \gamma_1' \in \Gamma_1 \) such that \( \rH_1(\gamma_1') \leq \refC{lift-mult} n^{\refC{lift-exp}} \) and \( \gamma_1 \gamma_1'^{-1} \in \Gamma(n) \cap \Gamma_1 \).

To conclude, let
\[ \gamma = \gamma' g \gamma_1 \gamma_1'^{-1} g^{-1}. \]
Because \( \gamma_1 \gamma_1'^{-1} \in \Gamma(n) \subset g^{-1} \Gamma g \), we have \( g \gamma_1 \gamma_1'^{-1} g^{-1} \in \Gamma \) and hence \( \gamma \in \Gamma \).

Since \( \Gamma_1 \) stabilises \( X_1^+ \) and acts trivially on \( X_2^+ \), we have
\[ z \in \gamma'.Y = \gamma' g(X_1^+ \times \{ x_2 \}) = \gamma' g \gamma_1 \gamma_1'^{-1} (X_1^+ \times \{ x_2 \}) = \gamma Y. \]

Finally, the height of \( \gamma' g \gamma_1 \) is polynomially bounded in terms of \( N(\pi(Y)) \) by~\eqref{eqn:ggg}.
By definition, \( \rH_1(\gamma_1') \) (and hence also \( \rH_1(\gamma_1'^{-1}) \)) is polynomially  bounded in terms of~\( n \) and hence in terms of \( \rH_1(g) \).
By \cref{hf-fund-desc}, \( \rH_1(g) \) (and hence also \( \rH_1(g^{-1}) \)) is polynomially bounded in terms of \( N(\pi(Y)) \).
Combining these facts yields the bound for the height of \( \gamma \).
\end{proof}

At last we are ready to prove the following special case of \cite[Conjecture~12.2]{DR} (modified for our definition of complexity).

\begin{proposition} \label{12.2}
Let \( (\gG, X^+) \), \( S \), $\rho$, \( (\gH, X_\gH^+, \gH_1, \gH_2) \), \( \cF \), \( \cF_1 \) and \( \cF_2 \) be as in \cref{hf-fund-desc}.

Then there exist constants \( d \), \( \newC{conj122-det-mult} \) and \( \newC{conj122-det-exp} \) such that, if \( Y \) is a pre-special subvariety in the Hecke--facteur family associated with \( (\gH, X_\gH^+, \gH_1, \gH_2) \) satisfying \( Y \cap \cF \neq \emptyset \), then there exists \( g\in \gG(\bQ)_+ \) and a (pre-special) point \( x \in X^+ \) such that \( Y = g \gH_1(\bR)^+ g^{-1}x \) and
\[ \rH_d(g, x)  \leq  \refC{conj122-det-mult} \, \Delta(\pi(Y))^{\refC{conj122-det-exp}}. \]
\end{proposition}

\begin{proof}
As explained in \cite[section~1.2]{DO16}, there is a constant \( d \) depending only on \( (\gG, X^+) \) such that the pre-special point \( x \) is defined over a number field of degree at most~\( d \).
Because \( g \in \gG(\bQ)_+ \), the height \( \rH_d(g, x) \) is finite for this value of~\( d \).
If \( (\gG, X^+) = (\gGSp_{2g}, \Hg) \), and $\rho$ is the inclusion $\gGSp_{2g}\hookrightarrow\gGL_{2g}$, then the theory of complex multiplication of abelian varieties tells us that we can take \( d = 2g \) (in particular, in the \( E \, \times \! \) CM case, we can take \( d = 4 \)).

Write \( Y = Y_{g,x_2} \) as in \cref{hf-fund-desc}.

Let $x'\in Y$ denote a pre-special point such that $\pi(x')$ has minimal complexity in $\pi(Y)$. Then we may choose $\gamma'\in\Gamma$ such that \( \gamma' x' \in \cF \).
By \cite[Theorems 1.1 and~4.1]{DO16}, the height $\rH_d(\gamma' x')$ is polynomially bounded in terms of $\Delta(\pi(x')) \leq \Delta(\pi(Y))$. 

Since \( \gamma' x' \in \cF \cap \Gamma Y \), by Proposition \ref{conj127-det}, there exists $\gamma\in\Gamma$ with height polynomially bounded in terms of $N(\pi(Y)) \leq \Delta(\pi(Y))$, such that
\[\gamma' x'\in\gamma Y.\]
Therefore, $x=\gamma^{-1}\gamma' x'$
has height polynomially bounded in terms of \( \Delta(\pi(Y)) \).
Meanwhile \( x \in Y \) and so
\[Y=g\gH_1(\RR)^+g^{-1} x.\]
This concludes the proof.
\end{proof}

\section{Proof of Theorems \ref{main-thm} and \ref{main-thm-gen}} \label{sec:main-proof}

\Cref{main-thm} is a special case of \cref{main-thm-gen}.
The proof of \cref{main-thm-gen} follows the same lines as that of \cite[Theorem~14.2]{DR}.
However, \cref{main-thm-gen} is not directly a corollary of \cite[Theorem~14.2]{DR} for three reasons:
\begin{enumerate}
\item we have used a different definition of complexity;
\item \cite[Theorem~14.2]{DR} requires \cite[Conjectures 11.1 and~12.2]{DR} to hold for all optimal points in \( V \), while we have proved versions of these conjectures only for points in the \( E \, \times \! \) CM Hecke--facteur family;
\item the proof of \cite[Theorem~14.2]{DR} uses the property that $\Delta(Z)=\Delta(\sigma(Z))$, which holds for the definition of complexity used in \cite{DR} (see \cite[Theorem~1.3]{OrrSV}) but not for our definition.
\end{enumerate}
We will, therefore, outline the modifications to the proof of \cite[Theorem~14.2]{DR} required in order to obtain \cref{main-thm-gen}.
For the convenience of the reader, we recall the statement of \cref{main-thm-gen} here.

\begin{theorem*} (\cref{main-thm-gen})
Let \( (\gG, X^+) \) be a Shimura datum component and let \( S = \Gamma \bs X^+ \) be an associated Shimura variety component.
Let $V$ be an irreducible algebraic curve in $S$ not contained in any proper special subvariety.

Let \( (\gH, X_\gH^+, \gH_1, \gH_2) \) be a facteur datum in \( (\gG, X^+) \) and let \( \fF \) be the associated Hecke--facteur family of special subvarieties of \( S \).
Suppose that, for any subvariety $Z \in \fF$, we have
\[\dim(Z)\leq\dim(S)-2.\]
Let \( \Sigma \) be the set of points \( s \in S \) for which the smallest special subvariety containing \( s \) is a member of $\fF$.

Assume that \( S\), $\fF$, $V$ and $\Sigma$ satisfy \cref{galois-orbits-conj-gen}.

Then \( V \cap \Sigma \) is finite.
\end{theorem*}

\begin{proof}
Fix a faithful representation $\gG\to\gGL_{m,\bQ}$ (which we use to define heights and complexities).
Because we only consider points whose special closure is a member of the Hecke--facteur family \( \fF \), the set \( \Omega \) used in the proof of \cite[Theorem~14.2]{DR} can be replaced by the one-element set \( \{ \gH_1 \} \). Let \( \cF \) be a fixed fundamental set as in \cref{fundamental-sets}. The constants $d$, $c_\mathcal{F}$ and $\delta_\mathcal{F}$ should be replaced by those afforded to us by \cref{12.2} and the constants $c_G$ and $\delta_G$ should be replaced by those afforded to us by \cref{galois-orbits-conj-gen}. As in \cite{DR}, $L$ is a finitely generated field of definition for $V$. We can and do assume that $L$ contains the field~$F$ afforded to us by \cref{galoisonhecke}.

Consider a point $P\in V\cap\Sigma$, and let $Z$ denote the smallest special subvariety of $S$ containing $P$. By hypothesis, $Z$ is a member of $\fF$. For each $\sigma\in\Aut(\CC/L)$, let
\[z_\sigma\in\mathcal{V}\cap\pi^{-1}(\sigma(P)) \subset X^+,\]
where $\mathcal{V}:=\pi^{-1}(V)\cap\cF$.
Let $Y_\sigma$ denote the smallest pre-special subvariety of $X^+$ containing $z_\sigma$.
Then, by Remark \ref{conjugates}, $\pi(Y_\sigma)=\sigma(Z)$, and so, by \cref{galoisonhecke}, $\pi(Y_\sigma)$ belongs to $\fF$, for all $\sigma\in\Aut(\CC/L)$.

By \cref{12.2} (which replaces \cite[Conjecture~12.2]{DR}),
\[Y_\sigma=g_\sigma\gH_1(\RR)^+g^{-1}_\sigma x_\sigma,\]
for $g_\sigma\in\gG(\QQ)_+$ and $x_\sigma\in X^+$ with \( d \)-height polynomially bounded in terms of \( \Delta(\pi(Y_\sigma)) = \Delta(\sigma(Z))\).
By \cref{galois-orbits-conj-gen} (which replaces \cite[Conjecture~11.1]{DR}), \( \Delta(\sigma(Z)) \) is bounded above by a polynomial in
\[ A := \# \Aut(\bC/L) \cdot \sigma(P)=\# \Aut(\bC/L) \cdot P.\]
Thus, the set \( \Sigma \) used in the proof of \cite[Theorem~14.2]{DR} (which is not the set $\Sigma$ used above) consists of tuples \( (g_\sigma, x_\sigma, z_\sigma) \) which satisfy \( H_d(g_\sigma, x_\sigma) \leq \newC* A^{\newC*} \). Since \( \# \pi_2(\Sigma)=A \) (where $\pi_2$ projects on to the last co-ordinate),
this is sufficient to apply the counting theorem \cite[Theorem~9.1]{DR} (a variant of \cite[Corollary~7.2]{HP16}).
The rest of the proof proceeds identically to that of \cite[Theorem~14.2]{DR}. We just make two remarks regarding the arguments appearing on p.1879 of \cite{DR}:
\begin{enumerate}
\item the equality $\dim(Y_1)=\dim\langle P_0\rangle$ at line 20 is true, by Remark \ref{conjugates}.
In fact we need only the easier inequality $\dim(Y_1)\leq\dim(X)-2$, which holds because $Y_1$ is a symmetric space for $F(\RR)^+$ and so has the same dimension as any member of $\fF$;
\item the equality $\pi_2(F)=G_2$ at line --5 is not difficult to prove, but is non-trivial.
\end{enumerate}
We conclude that $A$ is bounded, so $\Delta(Z)$ is bounded. This implies the theorem.
\end{proof}

\begin{remark}\label{conjugates}
The arguments of \cite{DR} rely on the assertion appearing in that article (before Conjecture~11.1) that, for any $\sigma\in\Aut(\CC/E_S)$ and any special subvariety $Z$ of $S$, the subvariety $\sigma(Z)$ of $S$ is special and $\Delta(\sigma(Z))=\Delta(Z)$ (for the definition of \( \Delta \) used in \cite{DR}). These assertions are true, though not proved in \cite{DR}. They are possibly known to the experts but, to our knowledge, the first proofs can be found in \cite{OrrSV}.
\end{remark}

\begin{remark}
Note that we removed the need for the equality \( \Delta(\sigma(Z)) = \Delta(Z) \), used in \cite[Theorem~14.2]{DR}, or even any comparison between complexities of \( \sigma(Z) \) and \( Z \), by comparing \( H_d(g_\sigma, x_\sigma) \) directly with \( \# \Aut(\bC/L) \cdot \sigma(P) \), which is trivially equal to \( \# \Aut(\bC/L) \cdot P \).
\end{remark}

\section{André's height bound} \label{sec:andre}

The key step in our proof of a special case of the large Galois orbits conjecture, as required for \cref{galois-orbits-thm}, is a generalisation of a height bound due to André for fibres with large endomorphism rings in an abelian scheme over a curve with a point of purely multiplicative reduction.
This height bound is stated and proved in \cite[Ch.~X]{And89} for abelian schemes of odd relative dimension~\( g \).
We want to apply it for \( g=2 \), so we need a generalisation of the theorem.
It turns out that we can replace the condition that \( g \) is odd by the condition that the generic endomorphism algebra of the abelian scheme is a totally real field of odd degree.

The theorem is a height bound for fibres of the abelian scheme whose endomorphism algebra is large in the following sense.
\begin{definition}
An abelian variety \( A \) of dimension~\( g \) is \defterm{exceptional} if there is no injection from \( \End(A) \) to \( \rM_g(\bQ) \).
\end{definition}

\begin{theorem} \label{andre-height-bound}
Let \( V' \) be a smooth connected (not necessarily proper) curve over a number field \( K \).
Let \( s_0 \) be a closed point in \( V' \) and let \( V = V' \setminus \{ s_0 \} \).
Let \( \bar\eta \) be a geometric generic point of~\( V \).
Let \( h \) denote a Weil height on \( V' \).

Let \( \fA \to V \) be an abelian scheme with purely multiplicative reduction at \( s_0 \)
(that is, if \( \fA' \to V' \) is the connected N\'eron model of \( \fA \) over \( V' \) then the fibre \( \fA'_{s_0} \) is a torus).
Let \( g \) be the relative dimension of \( \fA \to V \).

Assume that:
\begin{enumerate}[(i)]
\item \( g > 1 \).
\item \( \End(\fA_{\bar\eta}) \otimes \bQ \) is a totally real field \( E \).
\item \( [E:\bQ] \) is odd.
\item The generic Mumford--Tate group of \( \fA \) is \( \bG_{m,\bQ}.\Res_{E/\bQ} \Sp_{2n,E} \), where \( n = g/[E:\bQ] \).
\end{enumerate} 

Then there exist constants \( \newC{andre-mult} \) and \( \newC{andre-exp} \) such that, for every point \( s \in V(\Qbar) \), if \( \fA_s \) is exceptional then
\[ h(s) \leq \refC{andre-mult} \, [K(s):K]^{\refC{andre-exp}}. \]
\end{theorem}

\begin{remark} 
Theorem \ref{andre-height-bound} can be found at \cite[Ch.~X, Theorem~1.3]{And89} with conditions (ii)--(iv) replaced by: \( g \) is odd and \( \fA_{\bar\eta} \) is simple.
These conditions, together with the multiplicative reduction condition, imply (ii)--(iv): (ii) is \cite[Ch.~X, Lemma~2.2]{And89}; (iii) holds because the degree of a totally real field acting on a simple abelian variety~\( A \) must divide \( \dim(A) \) \cite[Ch.~IV, \S 20, p.~191, Corollary]{Mum74}; (iv) is \cite[Ch.~X, 3.3 Sublemma]{And89}.
On the other hand, when \( g \) is even, simplicity of \( \fA_{\bar\eta} \) is not sufficient to imply conditions (ii)--(iv). When $g=2$, condition (iv) is automatic because of the classification in section \ref{sec:a2-subvars}. However, when $g=4$, condition (iv) is famously not automatic, as first observed by Mumford.
\end{remark}

\begin{remark} 
Theorem \ref{andre-height-bound} is certainly not as general as possible. In particular, condition (iii) can be relaxed. For example, Masser has recently obtained a similar theorem for abelian schemes whose generic fibre is a principally polarised, simple abelian surface with real multiplication \cite{Masser}.
\end{remark}

\begin{remark} 
Observe that, if the abelian scheme \( \fA \to V \) is principally polarised and the image of the associated map \( V \to \cA_g \) is Hodge generic, then \( \End(\fA_{\bar\eta}) = \bZ \) and conditions (ii)--(iv) are satisfied.
\end{remark}

\begin{remark} 
For \( g = 2 \), abelian surfaces of \( E \, \times \! \) CM type are exceptional with respect to~\( \bQ \): their endomorphism algebra is \( \bQ \times F \), \( F_1 \times F_2 \) or \( \rM_2(F) \) where \( F, F_1, F_2 \) are imaginary quadratic fields.
Each of these algebras contains a commutative subalgebra of dimension at least~\( 3 \) over \( \bQ \), so they do not embed into \( \rM_2(\bQ) \).
Abelian surfaces whose endomorphism algebra is a non-split quaternion algebra over \( \bQ \) are also exceptional.
However abelian surfaces isogenous to the square of a non-CM elliptic curve are not exceptional (their endomorphism algebra is \( \rM_2(\bQ) \)).
\end{remark}

\begin{remark} \label{rmk:neron-model}
We can replace the ``multiplicative reduction'' condition in \cref{andre-height-bound} by the condition: there exists a semiabelian scheme \( \fA' \to V' \) such that \( \fA'_{|V} \cong \fA \) and \( \fA_{s_0} \) is a torus.
This is because \cite[Ch.~7, Prop.~3]{BLR90} tells us that, if \( \fA' \) is a semiabelian scheme whose generic fibre is an abelian variety, then \( \fA' \) is its own connected Néron model.
\end{remark}

Before explaining how to modify the proof of \cite[Ch.~X, Theorem~1.3]{And89} to obtain  \cref{andre-height-bound}, we prove the properties of endomorphism algebras of exceptional abelian varieties which we will use.

\begin{lemma} \label{quat-alg-cm}
Let \( D \) be a quaternion algebra over a totally real field \( E \).
(We allow \( D \) to be split.)
Let \( \dag \) be a positive involution of the \( \bQ \)-algebra \( D \).
Then there exists \( \alpha \in D^\times \) such that \( \alpha^\dag = -\alpha \) and \( E(\alpha) \) is a CM field.
\end{lemma}

\begin{proof}
First suppose that \( D \) is non-split over \( E \).

A positive involution must act trivially on the totally real field \( E \), so \( \dag \) is \( E \)-linear.
Since it is an involution, it is diagonalisable with eigenvalues \( \pm 1 \).
A positive involution always acts non-trivially on a quaternion algebra, so we deduce that
\[ \{ \alpha \in D : \alpha^\dag = -\alpha \} \]
is a non-zero \( E \)-linear subspace of \( D \).

Choose any non-zero element \( \alpha \) of this subspace.
The subalgebra \( E(\alpha) \subset D \) is a field because \( D \) is a division algebra.
Because \( \alpha^\dag = -\alpha \), \( \dag \) restricts to a non-trivial involution of the field \( E(\alpha) \).
This restriction remains a positive involution, and it is well-known that any field with a non-trivial positive involution is a CM field.

Now consider the case where \( D \) is split.
We may choose an isomorphism \( D \to \rM_2(E) \) under which \( \dag \) becomes either (a) matrix transposition or (b) the involution \( A \mapsto (\det A)A^{-1} \).
In either case, \( \alpha = \fullsmallmatrix{0}{1}{-1}{0} \) satisfies \( \alpha^\dag = -\alpha \).
Since \( \alpha^2 = -1 \), the algebra \( E(\alpha) \subset D \) is an extension of the totally real field \( E \) obtained by adjoining a square root of a totally negative element of \( E \).
Thus \( E(\alpha) \) is a CM field.
\end{proof}

\begin{lemma} \label{Ehat-E}
Let \( A \) be an abelian variety of dimension~\( g \).
Let \( \hat{E} \) be a maximal commutative subalgebra of \( \End(A) \otimes \bQ \).
If \( A \) is exceptional, then \( \hat{E} \) is not a field of odd degree over~\( \bQ \).
\end{lemma}

\begin{proof}
Assume for contradiction that \( \hat{E} \) is a field of odd degree.
Since \( \End(A) \otimes \bQ \) contains a maximal commutative subalgebra which is a field, \( \End(A) \otimes \bQ \) must be a simple algebra so it is isomorphic to \( \rM_r(D) \) for some division algebra~\( D \) and positive integer~\( r \).

We split into cases depending on the type of \( D \) in the Albert classification.

If \( D \) has type I (a totally real field), then \( [\hat{E}:\bQ] = r[D:\bQ] \).
But \( [\hat{E}:\bQ] \) divides~\( g \).
Hence, \( \rM_r(D) \) injects into \( \rM_g(\bQ) \), contradicting the exceptionality of~\( A \).

If \( D \) has type II or III (a quaternion algebra) or type IV (a division algebra whose centre is a CM field), then any maximal commutative subalgebra of \( \rM_r(D) \) has even degree over~\( \bQ \), contradicting our hypothesis.
\end{proof}

\begin{lemma} \label{Ehat-E2}
Let \( E \) be a totally real field of odd degree.
Let \( A \) be a polarised abelian variety of dimension \( g \), equipped with an embedding \( E \to \End(A) \otimes \bQ \) whose image is stabilised by the Rosati involution.
If \( A \) is exceptional and the degree of a maximal commutative subalgebra of \( \End(A) \otimes \bQ \) is \( 2[E:\bQ] \), then there exists a maximal commutative subalgebra of \( \End(A) \otimes \bQ \) which contains \( E \), is a CM field, and is stabilised by the Rosati involution.
\end{lemma}

\begin{proof}
First we show that \( A \) is isotypic.
Suppose not.
Then \( \End(A) \otimes \bQ \) is not simple, so its maximal commutative subalgebras are never fields.
Since \( A \) has a maximal commutative subalgebra of degree~\( 2 \) over \( E \), it must have \( E \times E \) as a maximal commutative subalgebra.
Therefore, \( A \) is isogenous to a product \( A_1 \times A_2 \), where \( \End(A_1) \otimes \bQ \) and \( \End(A_2) \otimes \bQ \) both have \( E \) as a maximal commutative subalgebra.
By \cref{Ehat-E}, \( A_1 \) and \( A_2 \) are not exceptional.
In other words, \( \End(A_1) \otimes \bQ \) injects into \( \rM_{g_1}(\bQ) \) and \( \End(A_2) \otimes \bQ \) injects into \( \rM_{g_2}(\bQ) \), where \( g_i = \dim(A_i) \).
Since \( \End(A) \otimes \bQ \cong (\End(A_1) \otimes \bQ) \times (\End(A_2) \otimes \bQ) \), we deduce that \( \End(A) \otimes \bQ \) injects into \( \rM_{g_1+g_2}(\bQ) \).
This contradicts the exceptionality of \( A \).

Thus \( A \) is isotypic.
Let \( Z \) denote the centraliser of \( E \) in \( \End(A) \otimes \bQ \).
Since \( E \) is commutative, \( E \subset Z \) and so by the double centraliser theorem, \( E \) is the centre of \( Z \).
Any maximal commutative subalgebra of \( \End(A) \otimes \bQ \) containing \( E \) is also a maximal commutative subalgebra of \( Z \); since such a maximal commutative subalgebra has degree~\( 2 \) over~\( E \), we deduce that \( Z \) is a quaternion algebra over \( E \)
(perhaps a split quaternion algebra).

Let \( \dag \) denote the Rosati involution on \( \End(A) \otimes \bQ \).
Since \( \dag \) stabilises \( E \), it also stabilises its centraliser \( Z \).
Applying \cref{quat-alg-cm}, we get \( \alpha \in Z^\times \) such that \( \alpha^\dag = -\alpha \) and \( E(\alpha) \) is a CM field.
Now \( [E(\alpha):\bQ] = 2[E:\bQ] \) is a maximal commutative subalgebra of \( \End(A) \otimes \bQ \).
Furthermore the fact that \( \alpha^\dag = -\alpha \) implies that \( \dag \) stabilises \( E(\alpha) \).
Thus \( E(\alpha) \) is the desired maximal commutative subalgebra.
\end{proof}

We now describe the modifications to the proof of \cite[Ch.~X, Theorem~1.3]{And89} required to obtain the more general \cref{andre-height-bound}.

\begin{proof}[Proof of \cref{andre-height-bound}]
As mentioned above, conditions (ii) and~(iv) in \cref{andre-height-bound} are precisely the conclusions of \cite[Ch.~X, Lemma~2.2 and Sublemma~3.3]{And89}.
Other than in these lemmas, the only place where \cite{And89} uses the oddness of \( g \) is in Construction~2.4.1, so this is the only part of the proof which we need to modify.

First we recall some of the notation from \cite[Ch.~X]{And89}.
Let \( s \) be a point in \( V(\Qbar) \) such that \( \fA_s \) is exceptional.
Let \( E = \End(\fA_{\bar\eta}) \otimes \bQ \) (we are assuming that this is a totally real field of odd degree) and let \( \hat E \) denote a maximal commutative subalgebra of \( \End(\fA_s) \otimes \bQ \) which contains \( E \).
Let \( n = g/[E:\bQ] \).

Let \( f \colon \fA \to V \) denote the structure map of our abelian scheme.
Let \( F \) be the Galois closure of \( E \) and let \( \hat{F} \) be the compositum of the Galois closures of the simple factors of \( \hat{E} \).
Then the local system of \( F \)-vector spaces \( R_1 f^{\mathrm{an}}_{\bC,*}(F) \) splits as a direct sum \( \bigoplus_{\sigma \colon E \to \bC} W_\sigma \), where \( E \) acts on \( W_\sigma \) via \( \sigma \).
For each \( \sigma \), \( W_\sigma \) has \( F \)-rank~\( 2n \).
Similarly \( H_1(\fA_s(\bC), \hat{F}) \) splits as a direct sum \( \bigoplus_{\hat\sigma \colon \hat{E} \to \bC} \hat{W}_{\hat\sigma} \).
These splittings are compatible in the sense that
\[ (W_\sigma)_s \otimes_F \hat{F} = \bigoplus_{\hat\sigma | \sigma} \hat{W}_{\hat\sigma}. \]

Replacing \( K \) by a finite extension and \( V \) by a Zariski open subset of an étale cover, we may assume that \( F \subset K \), \( \End(\fA_\eta) = \End(\fA_{\bar\eta}) \) (where \( \eta \) denotes the generic point of \( V \) as a \( K \)-variety) and the \( \cO_V \)-module \( H^1_{DR}(\fA/V) \) is free.
Let \( \hat K \) denote the compositum \( K(s) \hat F \).

Similarly to the above, we get splittings
\begin{gather*}
H^0(V, H^1_{DR}(\fA/V)) \otimes_{\cO(V)} K(V) = \bigoplus_{\sigma \colon E \to K(V)} W_{DR}^\sigma,
\\ H^1_{DR}(\fA_s / \hat K) = \bigoplus_{\hat\sigma \colon \hat E \to \hat K} \hat W_{DR}^{\hat\sigma}.
\end{gather*}

Choose a polarisation of the abelian scheme \( \fA \to V \).
Thanks to the isomorphism
\[ R^2 f^{\mathrm{an}}_{\bC,*}(\bQ(1)) \cong \Hom\bigl( \bigwedge\nolimits^2 R_1 f^{\mathrm{an}}_{\bC,*}(\bZ), \bQ(1) \bigr), \]
this polarisation gives rise to a symplectic form \( 2\pi i\langle,\rangle_\bQ \) on \( R_1 f^{\mathrm{an}}_{\bC,*}(\bQ) \) with values in \( \bQ(1) \).
Because \( E \) is totally real, this form (after extension of scalars to \( F \)) is a sum of symplectic forms on the components \( W_\sigma \) which we denote by
\[ 2\pi i \langle,\rangle_\sigma \colon W_\sigma \times W_\sigma \to F(1). \]

Let \( \Delta \) be the image of a holomorphic embedding of the unit disc in \( V'(\bC) \) mapping \( 0 \) to \( s_0 \), let \( \Delta^* = \Delta \setminus \{ s_0 \} \) and let \( U \) be a simply connected dense open subset of \( \Delta^* \).
Note that we can cover \( V(\bC) \) by finitely many such sets~\( U \), so we may assume that our exceptional point \( s \) is in \( U \).

Let \( W^1_\sigma \) denote the maximal constant subsystem of \( W_{\sigma|\Delta^*} \).
As a consequence of the multiplicative reduction condition at \( s_0 \), \( W^1_\sigma \) is totally isotropic with respect to \( \langle,\rangle_\sigma \) and has \( F \)-rank~\( n \) \cite[Ch.~X, Lemma~2.3]{And89}.

Choose bases \( \gamma_{\tau,1}, \dotsc, \gamma_{\tau,n} \) for \( W^1_\tau \) and \( \omega_{\sigma,1}, \dotsc, \omega_{\sigma,2n} \) for \( W_{DR}^\sigma \).
The ``locally invariant periods'' of \( \fA/V \) are the meromorphic functions on \( \Delta^* \) defined by
\[ t \mapsto \int_{\gamma_{\tau,j}(t)} \omega_{\sigma,i}(t) \]
(for \( \sigma \colon E \to K(V) \), \( \tau \colon E \to \bC \), \( 1 \leq i \leq 2n \), \( 1 \leq j \leq n \)).
These are the G-functions which lie at the heart of the proof of \cref{andre-height-bound}.

\medskip

We now define symplectic forms \( \langle,\rangle_{DR}^\sigma \colon W_{DR}^\sigma \times W_{DR}^\sigma \to \hat K \) which are dual to the forms \( 2\pi i \langle,\rangle_\sigma \) in a sense which we will make precise below.

Let \( \xi \) denote the class in \( H^2_{DR}(\fA_\eta/K(V)) \) induced by our chosen polarisation of~\( \fA \).
Since \( \dim_{K(V)} H^{2g}_{DR}(\fA_\eta/K(V)) = 1 \) and \( \xi^g \neq 0 \), the following formula defines a symplectic form \( \langle,\rangle_{DR} \) on \( H^1_{DR}(\fA_\eta/K(V)) \) with values in \( K(V) \):
\[ \alpha \cup \beta \cup \xi^{g-1} = \frac{1}{g} \, \langle \alpha,\beta \rangle_{DR} \; \xi^g \in H^{2g}_{DR}(\fA_\eta/K(V)). \]

Let \( \mathcal{M}(U) \) denote the field of meromorphic functions on~\( U \).
Write \( R_1 f_*(\bQ(1)) \otimes_\bQ \mathcal{M}(U) \) as short hand for \( H^0(U, R_1 f^{\mathrm{an}}_{\bC,*}(\bQ(1))) \otimes_\bQ \mathcal{M}(U). \)
The local system \( W_{\sigma|U} \) is trivial because \( U \) is simply connected, so the pairing \( (\gamma, \omega) \mapsto \int_\gamma \omega \) (in each fibre of \( \fA_{|U} \to U \)) induces an isomorphism
\begin{equation} \label{eqn:de-rham-isom}
R_1 f_*(\bQ(1)) \otimes_\bQ \mathcal{M}(U) \to (H^1_{DR}(\fA_\eta/K(V)) \otimes_{K(V)} \mathcal{M}(U))^\vee.
\end{equation}
Under this isomorphism, the symplectic form \( 2\pi i \langle,\rangle_\bQ \in \bigwedge^2 (R_1 f_*(\bQ(1)) \otimes_\bQ \mathcal{M}(U))^\vee \) corresponds to \( \xi \in H^2_{DR}(\fA_\eta/K(V)) = \bigwedge^2 H_{DR}^1(\fA_\eta/K(V)) \).
A calculation shows that, if \( e_1, \dotsc, e_{2g} \) is a symplectic basis for \( R_1 f^{\mathrm{an}}_{\bC,*}(\bQ(1))_{|U} \) with respect to \( 2\pi i\langle,\rangle_\bQ \) and \( e_1^\vee, \dotsc, e_{2g}^\vee \) is the corresponding dual basis for \( H_{DR}^1(\fA_\eta/K(V)) \otimes_{K(V)} \mathcal{M}(U) \), then \( e_1^\vee, \dotsc, e_{2g}^\vee \) is a symplectic basis with respect to \( \langle,\rangle_{DR} \).

Since \( 2\pi i\langle,\rangle_\bQ \) is the orthogonal sum of symplectic forms on \( W_\sigma \) and the isomorphism~\eqref{eqn:de-rham-isom} is compatible with the actions of \( E \), we deduce that \( \langle,\rangle_{DR} \) is the orthogonal sum of symplectic forms
\[ \langle,\rangle_{DR}^\sigma \colon W_{DR}^\sigma \times W_{DR}^\sigma \to \hat K. \]

\medskip

The notational set-up is complete.
We now explain the modifications of \cite[Ch.~X, Construction~2.4.1]{And89} needed to obtain non-trivial global relations between the values of the locally invariant periods at a point \( s \in V(\ov\bQ) \) where \( \fA_s \) is exceptional.

The initial remarks of \cite[Ch.~X, Construction~2.4.1]{And89} show that \( \hat E \neq E \).
The argument in \cite{And89} relies on \( g \) being odd, but we can replace this by \cref{Ehat-E}.

If \( [\hat E:E] \geq 3 \), then we are in Case~1 of \cite[Ch.~X, Construction~2.4.1]{And89}:
for each \( \sigma \colon E \to \bC \) there are at least three \( \hat\tau \colon \hat{E} \to \bC \) extending \( \sigma \).
Hence at least one of these \( \hat\tau \) must satisfy \( \dim_{\hat{F}} (\hat{W}_{\hat\tau}) < \frac{1}{2} \dim_F ((W_\sigma)_s) \).
The proof for Case~1 in \cite{And89} does not use the oddness of~\( g \), so it works without change.

If \( [\hat E:E] = 2 \), then we may choose \( \hat E \) according to \cref{Ehat-E2}.
For \( \sigma \colon E \to \bC \), let \( \hat\sigma_1 \) and \( \hat\sigma_2 \) denote the two homomorphisms \( \hat E \to \bC \) extending \( \sigma \).
An element of \( \Aut(\bC) \) which exchanges \( \hat\sigma_1 \) and \( \hat\sigma_2 \) will exchange \( \hat W_{\hat\sigma_1} \) and \( \hat W_{\hat\sigma_2} \) so
\[ \dim_{\hat F} (\hat W_{\hat\sigma_1}) = \dim_{\hat F} (\hat W_{\hat\sigma_2}) = \half \dim_F ((W_{\sigma})_s). \]
Thus we are in either Case 2 or Case~3 of \cite[Ch.~X, Construction~2.4.1]{And89}.

Case~2 occurs when
\( \hat W_{\hat\sigma_2} \cap \bigl( (W^1_\sigma)_s \otimes_F \hat F \bigr)  \neq  0 \).
Again the proof in \cite{And89} does not use the oddness of~\( g \), so no change is required.

Finally we consider Case~3, where
\( \hat W_{\hat\sigma_2} \cap \bigl( (W^1_\sigma)_s \otimes_F \hat F \bigr)  =  0 \).
In the setting of \cite{And89}, the oddness of~\( g \) implies that \( \hat E \) is a CM field.
In our more general setting, it might not be true that every choice of maximal commutative subalgebra \( \hat E \subset \End(\fA_s) \otimes \bQ \) containing \( E \) is a CM field, but \cref{Ehat-E2} ensures that it is possible to choose an \( \hat E \) which is a CM field.

The fact that \( \hat E \) is a CM field stabilised by the Rosati involution is crucial for the construction of period relations.
Indeed, the Rosati involution restricts to complex conjugation on \( \hat E \) and so \( \hat W_{\hat\sigma_2} \) is isotropic with respect to~\( \langle,\rangle_\sigma \).
This ensures that we can choose a basis \( \gamma_{\sigma,n+1}, \dotsc, \gamma_{\sigma,2n} \) for \( \hat W_{\hat\sigma_2} \) which, when combined with the previously chosen basis \( \gamma_{\sigma,1}(s), \dotsc, \gamma_{\sigma,n}(s) \) for \( (W^1_\sigma)_s \), forms a symplectic basis for \( (W_\sigma)_s \otimes_F \hat F \).
Similarly, the fact that \( \hat E \) is a CM field stabilised by the Rosati involution implies that we can choose a symplectic basis \( \hat\omega_{\sigma,1}, \dotsc, \hat\omega_{\sigma,2n} \) for \( W_{DR}^\sigma \) in such a way that \( \hat\omega_{\sigma,1}(s), \dotsc, \hat\omega_{\sigma,n}(s) \in \hat W_{DR}^{\hat\sigma_1} \) and \( \hat\omega_{\sigma,n+1}(s), \dotsc, \hat\omega_{\sigma,2n}(s) \in \hat W_{DR}^{\hat\sigma_2} \).

Because these bases for \( (W_\sigma)_s \otimes_F \hat F \) and \( W_{DR}^\sigma \) are symplectic, the resulting period matrix satisfies the Riemann relations \cite[Ch.~X, (2.3.1)--(2.3.3)]{And89}.
The remainder of \cite[Ch.~X, Construction~2.4.1, Case~3]{And89}, using these relations to obtain a quadratic relation between locally invariant periods at \( s \), does not use the oddness of~\( g \) so we can reuse it without change.
\end{proof}

\section{Galois orbits} \label{sec:galois-orbits}

In this section we prove special cases of our Galois orbits conjecture and deduce \cref{galois-orbits-thm,galois-orbits-thm-sch}.
The proof of \cref{galois-orbits-thm-sch} combines \cref{andre-height-bound} with the Masser--W\"ustholz isogeny theorem and a refinement of the Brauer--Siegel theorem.
To deduce \cref{galois-orbits-thm} from \cref{galois-orbits-thm-sch}, we use a toroidal compactification of \( \Ag \) to construct a semiabelian scheme to which we can apply \cref{galois-orbits-thm-sch}.
We also show that the two formulations of the Galois orbits conjecture, \cref{galois-orbits-conj,galois-orbits-conj-sch} are equivalent.

Given an abelian surface scheme \( \fA \to V \), recall that we defined \( N(s) \) and \( \Delta'(s) \) for \( E \, \times \! \) CM points \( s \in V \) in \cref{delta's}.

\begin{proposition} \label{galois-orbits-equiv}
\Cref{galois-orbits-conj,galois-orbits-conj-sch} are equivalent.
\end{proposition}

\begin{proof}
Firstly let \( \fA \to V \) be an abelian scheme as in \cref{galois-orbits-conj-sch}, defined over a finitely generated field \( L \subset \bC \).
This abelian scheme induces a morphism of varieties \( q \colon V \to \cA_2 \), also defined over~\( L \).
Since \( \fA \to V \) is not isotrivial, the image \( q(V) \) is still a curve.
Since \( \End(\fA_{\ov\eta}) = \bZ \), \( q(V) \) is not contained in any proper special subvariety of \( \cA_2 \) (see section~\ref{sec:a2-subvars}).
A point \( s \in V \) is an \( E \, \times \! \) CM point if and only if \( q(s) \in \Sigma \).
Hence we can apply \cref{galois-orbits-conj} to the Zariski closure of \( q(V) \) in~\( \cA_2 \), obtaining a bound
\[ \# \Aut(\bC/L) \cdot q(s)  \geq  \newC* \Delta(Z)^{\newC*} \]
for all \( E \, \times \! \) CM points \( s \in V \), where \( Z \) is the unique \( E \, \times \! \) CM curve containing \( q(s) \).
Using \cref{comp-poly-equiv} and the fact that \( \# \Aut(\bC/L) \cdot s \geq \# \Aut(\bC/L) \cdot q(s) \), we deduce that \cref{galois-orbits-conj} holds.

Conversely, let \( V \subset \cA_2 \) be an algebraic curve as in \cref{galois-orbits-conj}, defined over a finitely generated field \( L \subset \bC \).
Let \( \tilde{V} \) be an irreducible component of the preimage of \( V \) in \( \cA_{2,3} \).
This is defined over a finite extension \( \tilde{L} \) of~\( L \).
The universal abelian scheme over \( \cA_{2,3} \) restricts to an abelian scheme \( \tilde\fA \to \tilde{V} \).
Since \( V \) is not contained in any special subvariety of \( \cA_2 \), \( \End(\tilde\fA_{\ov\eta}) = \bZ \).
If \( s \in V \cap \Sigma \), then we can find an \( E \, \times \! \) CM point \( \tilde{s} \in \tilde{V} \) such that \( q(\tilde{s}) = s \).
Thanks to \cref{galois-orbits-conj-sch}, this satisfies
\[ \# \Aut(\bC/\tilde{L}) \cdot \tilde{s} \geq \newC* \Delta'(\tilde{s})^{\newC*}. \]
Because \( q \) is a finite morphism,
\[ \# \Aut(\bC/L) \cdot s \geq \newC* \, \# \Aut(\bC/\tilde{L}) \cdot \tilde{s}. \]
We can therefore use \cref{comp-poly-equiv} to complete the proof of \cref{galois-orbits-conj-sch}.
\end{proof}

As a first step towards the proof of \cref{galois-orbits-conj-sch} under the conditions of \cref{galois-orbits-thm-sch}, we prove a Galois bound relative to \( N(s) \).
We then combine this with a lower bound for class numbers of imaginary quadratic fields.

\begin{proposition} \label{galois-degree-bound}
Let \( V \) be an irreducible algebraic curve over \( \Qbar \) and let \( \fA \to V \) be a principally polarised non-isotrivial abelian scheme of relative dimension~\( 2 \).
Suppose that \( \End(\fA_{\ov\eta}) = \bZ \), where \( \ov\eta \) denotes a geometric generic point of \( V \).

Suppose also that there exist a curve \( V' \), a semiabelian scheme \( \fA' \to V' \) and an open immersion \( \iota \colon V \to V' \) (all defined over \( \Qbar \)) such that \( \fA \isom \iota^* \fA' \) and there is some point \( s_0 \in V'(\ov\bQ) \) for which the fibre \( \fA'_{s_0} \) is a torus.

Let \( L \) be a number field over which \( V \) and \( \fA \to V \) are defined.

Then there exist positive constants \( \newC{galois-degree-mult} \) and \( \newC{galois-degree-exp} \) such that, for every \( E \, \times \! \) CM point \( s \in V \),
\[ [L(s):L] \geq \refC{galois-degree-mult} \, N(s)^{\refC{galois-degree-exp}}. \]
\end{proposition}

\begin{proof}
After a finite extension of \( L \), we may assume that \( V' \), \( \fA' \to V' \), \( \iota \colon V \to V' \) and \( s_0 \) are all defined over \( L \).

Because \( \End(\fA_{\ov\eta}) = \bZ \) and \( \dim(\fA_{\ov\eta}) = 2 \), the Mumford--Tate group of \( \fA_{\ov\eta} \) is \( \gGSp_{4,\bQ} \) (see section~\ref{sec:a2-subvars}).
Thus \( \fA \to V \) satisfies the conditions of \cref{andre-height-bound}, as modified in \cref{rmk:neron-model}.

Let \( s \) be an \( E \, \times \! \) CM point in \( V \).
The image of \( s \) under the map \( V \to \cA_2 \) induced by \( \fA \to V \) is in the intersection between the image of \( V \) and a special curve.
Since the image of \( V \) is a curve defined over a number field, we deduce that \( s \in V(\Qbar) \).

Now \( \End(\mathfrak{A}_s) \otimes \bQ \) contains \( \bQ \times F \) where \( F \) is an imaginary quadratic field.
This is a commutative \( \bQ \)-algebra of degree~\( 3 \), so cannot inject into \( \rM_2(\bQ) \).
Hence \( \fA_s \) is an exceptional abelian surface in the sense of \cref{andre-height-bound}.

Therefore by \cref{andre-height-bound}, \( h(s) \) is polynomially bounded in terms of \( [L(s):L] \), where \( h \) denotes a Weil height on \( V' \).
Let \( h_F \) denote the stable Faltings height.
As proved in \cite[p.~356]{Fal83},
\[ \abs{h_F(\fA_s) -  h(s)}  = \rO(\log  h(s)). \]
We conclude that \( h_F(\fA_s) \) is polynomially bounded in terms of \( [L(s):L] \).

By \cref{detstar-isog-comparison}, there exist elliptic curves \( E_1 \) and \( E_2 \) and a polarised isogeny \( E_1 \times E_2 \to \fA_s \) of degree~\( N(s) \), and \( N(s) \) is the minimum degree of any polarised isogeny \( E_1 \times E_2 \to \fA_s \).
Thanks to \cref{excm-polarised}, \( N(s) \) is the minimum degree of any isogeny \( E_1 \times E_2 \to \fA_s \) (polarised or not).
By \cref{min-isogeny-unique}, \( E_1 \) and \( E_2 \) are the only pair of elliptic curves whose product possesses an isogeny to~\( \fA_s \) of degree~\( N(s) \).
Therefore, for every \( \sigma \in \Aut(\bC/L(s)) \), \( E_1^\sigma \cong E_1 \) and \( E_2^\sigma \cong E_2 \) (\( E_2 \) has CM and \( E_1 \) does not, so \( \sigma \) cannot swap them).
Since elliptic curves are defined over their fields of moduli, we conclude that \( E_1 \) and \( E_2 \) are both defined over \( L(s) \).

Therefore by the Masser--W\"ustholz isogeny theorem \cite{MW93}, there exists an isogeny \( E_1 \times E_2 \to \fA_s \) of degree at most
\[ \newC* \, \max(h_F(\fA_s), [L(s):L])^{\newC*}. \]
We conclude that \( N(s) \) also satisfies this bound.
Combining this with the fact that \( h_F(\fA_s) \) is polynomially bounded in terms of \( [L(s):L] \) completes the proof.
\end{proof}

\begin{proposition} \label{galois-orbits-holds-sch}
Let \( \fA \to V \) be an abelian scheme satisfying the conditions of \cref{galois-degree-bound} (including the multiplicative reduction condition).

Then \cref{galois-orbits-conj-sch} holds for \( \fA \to V \).
\end{proposition}

\begin{proof}
Let \( L \) be a number field over which \( V \) and \( \fA \to V \) are defined.
Let \( s \in V \) be an \( E \, \times \! \) CM point.
By \cref{detstar-isog-comparison}, there exist elliptic curves \( E_1 \) and \( E_2 \) and an isogeny \( E_1 \times E_2 \to A_s \) of degree~\( N(s) \), where \( E_2 \) has CM and \( E_1 \) does not.
Then
\[ \Delta'(s) = \max \{ N(s), \abs{\disc(\End(E_2))} \}. \]
According to the argument from the proof of \cref{galois-degree-bound}, \( E_1 \) and \( E_2 \) are defined over \( L(s) \).
Therefore by \cite[Proposition~(1.8)]{Len87} (refining the Brauer--Siegel theorem), \( \abs{\disc(\End(E_2))} \) is polynomially bounded in terms of \( [L(s):L] \).
Meanwhile \cref{galois-degree-bound} tells us that \( N(s) \) is polynomially bounded in terms of \( [L(s):L] \).
\end{proof}

\begin{proof}[Proof of \cref{galois-orbits-thm-sch}]
Let \( \fA \to V \) be an abelian scheme satisfying the conditions of \cref{galois-orbits-thm-sch}.
This abelian scheme induces a morphism of varieties \( q \colon V \to \cA_2 \), also defined over~\( L \).
This morphism has finite fibres because \( V \) is a curve and \( \fA \to V \) is non-isotrivial.
Since \( \End(\fA_{\ov\eta}) = \bZ \), the image of \( V \) in \( \cA_2 \) is not contained in a proper special subvariety (see section~\ref{sec:a2-subvars}).

Hence \cref{main-thm} applies to the Zariski closure of the image of \( V \) in \( \cA_2 \).
Combining this with \cref{galois-orbits-holds-sch}, we deduce that the image of \( V \) has finite intersection with \( \Sigma \).
We conclude because a point in \( V \) is an \( E \, \times \! \) CM point if and only if its image lies in \( \Sigma \).
\end{proof}

If \( V \) is a curve satisfying the condition of \cref{galois-orbits-thm} involving the Baily--Borel compactification, the following proposition allows us to construct an abelian scheme over a finite cover of \( V \) which satisfies the multiplicative reduction condition of \cref{galois-orbits-holds-sch}, enabling us to deduce \cref{galois-orbits-thm} from \cref{galois-orbits-holds-sch}.

\begin{proposition} \label{mult-reduction}
Let \( V \subset \Ag \) be an irreducible algebraic curve such that the Zariski closure of \( V \) in the Baily--Borel compactification of \( \Ag \) intersects the zero-dimensional stratum of the compactification.
Then there exists a smooth projective curve \( \tilde{V}' \), an open subset \( \tilde{V} \subset \tilde{V}' \), a point \( s_0 \in \tilde{V}' \setminus \tilde{V} \), a finite surjective morphism \( q \colon \tilde{V} \to V \) and a semiabelian scheme \( \fA' \to \tilde{V}' \) such that \( \fA'_{|\tilde{V}} \) is an abelian scheme, the map to the coarse moduli space associated with \( \fA'_{|\tilde{V}} \to \tilde{V} \) is the composition \( \tilde{V} \overset{q}{\longrightarrow} V \hookrightarrow \Ag \), and the fibre \( \fA'_{s_0} \) is a torus.
\end{proposition}

\begin{proof}
Let \( \Ag^* \) be the Baily--Borel compactification of \( \Ag \).
Let \( V^* \) be the closure of \( V \) in \( \Ag^* \).
Recall that \( \Ag^* \) has a stratification
\[ \Ag^* = \Ag \sqcup \cA_{g-1} \sqcup \dotsb \sqcup \cA_1 \sqcup \cA_0. \]
By hypothesis, \( V^* \) intersects the stratum \( \cA_0 \).
This stratum is just a closed point, which we shall call \( s^* \).

Let \( \Agt \) denote the moduli space of principally polarised abelian varieties with full symplectic level-\( 3 \) structure.
Let \( \ov{\Ag} \) and \( \ov{\Agt} \) denote toroidal compactifications of \( \Ag \) and \( \Agt \) respectively, as defined in \cite[Ch.~IV, secs. 5 and~6]{FC90}.
We construct these toroidal compactifications with respect to some projective smooth admissible cone decomposition for \( \bZ^g \) (see \cite[Ch.~V, Definition~5.1]{FC90} and remark~(c) following that definition).
Then \( \ov{\Ag} \) is a proper algebraic stack \cite[Ch.~IV, Theorem~5.7]{FC90} and \( \ov{\Agt} \) is a projective algebraic variety \cite[Ch.~V, sec.~5]{FC90}.
Note that \cite{FC90} constructs \( \ov{\Ag} \) as a stack over \( \bZ \) and \( \ov{\Agt} \) as a scheme over \( \bZ[\zeta_3, 1/3] \), but we require them only over \( \bQ \) and \( \bQ(\zeta_3) \) respectively.

We have the following commutative diagram of morphisms of stacks over \( \bQ(\zeta_3) \):
\[ \xymatrix{
   \Agt       \ar[r] \ar[d]
 & \ov{\Agt}  \ar[d]^p
\\ \Ag        \ar[r]
 & \ov{\Ag}   \ar[r]^{\ov\pi}
 & \Ag^*
} \]
The composition \( \ov\pi \circ p \colon \ov{\Agt} \to \Ag^* \) is a morphism of stacks between varieties, and consequently it is a morphism of varieties.

Let \( V_3 \) be an irreducible component of the preimage of \( V \) in \( \Agt \) which surjects onto \( V \).
Let \( \ov{V_3} \) be the Zariski closure of \( V_3 \) in \( \ov{\Agt} \).
This is a projective variety and hence \( \ov\pi \circ p(\ov{V_3}) \) is a closed subset of \( \Ag^* \).
Furthermore, \( \ov\pi \circ p \) restricts to the natural map \( \Agt \to \Ag \) and hence \( \ov\pi \circ p(V_3) = V \).
Therefore \( s^* \in \ov\pi \circ p(\ov{V_3}) \).
Hence we can choose a closed point \( \ov{s_3} \in \ov{V_3} \) such that \( \ov\pi \circ p(\ov{s_3}) = s^* \).

By \cite[Ch.~IV, Theorem~5.7~(3)]{FC90}, there is a semiabelian scheme \( \fG \) over the stack \( \ov{\Ag} \) which extends the universal abelian scheme over \( \Ag \) (considered as an algebraic stack).
We can pull this back to a semiabelian scheme \( \fG_3 \to \ov{\Agt} \).
Define a stratification on \( \ov{\Ag} \) by the dimensions of the abelian part of the fibres of \( \fG \).
According to \cite[Ch.~V, Theorem~2.3~(5)]{FC90}, \( \ov\pi \) is compatible with the stratifications of \( \ov{\Ag} \) and \( \Ag^* \).
It follows that \( \fG_{3,\ov{s_3}} \) is a torus.

To finish the proof, let \( \tilde{V'} \) be the normalisation of \( \ov{V_3} \), let \( \tilde{V} \) be the preimage of \( V_3 \) in \( \tilde{V'} \), let \( s_0 \) be a preimage of \( \ov{s_3} \) in \( \tilde{V'} \) and let \( \fA' \) be the pullback of \( \fG_3 \) to~\( \tilde{V'} \).
\end{proof}

\begin{proposition} \label{go-sch-implies-go}
Let \( V \subset \cA_2 \) be an algebraic curve satisfying the conditions of \cref{galois-orbits-thm}.
Then \cref{galois-orbits-conj} holds for~\( V \).
\end{proposition}

\begin{proof}
Let \( L \) be a number field over which \( V \) is defined.
We can construct  \( \tilde{V}' \), \( q \colon \tilde{V} \to V \), \( s_0 \in \tilde{V}' \) and \( \fA' \) as in \cref{mult-reduction}.
We can find a finite extension \( \tilde{L}/L \) such that \( \tilde{V}' \), \( q \colon \tilde{V} \to V \), \( s_0 \) and \( \fA' \to \tilde{V}' \) are all defined over \( \tilde{L} \).
The abelian scheme \( \fA'_{|\tilde{V}} \to \tilde{V} \) and the point \( s_0 \in \tilde{V}'(\ov\bQ) \) satisfy the conditions of \cref{galois-orbits-holds-sch}.
We can therefore complete the proof using \cref{galois-orbits-equiv,galois-orbits-holds-sch}.
\end{proof}

Combining \cref{main-thm,go-sch-implies-go} proves \cref{galois-orbits-thm}.


\begin{thebibliography}{Mum74}

\bibitem[And89]{And89}
Y.~Andr\'e, ``G-functions and geometry'', Aspects of Mathematics, E13,
  Friedr. Vieweg \& Sohn, Braunschweig, 1989.

\bibitem[BHC62]{BHC62}
A.~Borel and Harish-Chandra, \emph{Arithmetic subgroups of algebraic groups},
  Ann. of Math. (2) \textbf{75} (1962), 485--535.

\bibitem[BLR90]{BLR90}
S.~Bosch, W.~L\"utkebohmert, and M.~Raynaud, ``N\'eron models'',
  Ergeb. Math. Grenzgeb. (3), vol.~21,
  Springer-Verlag, Berlin, 1990.

\bibitem[Bor84]{Bor84}
M.~V. Borovo\u\i, \emph{Langlands' conjecture concerning conjugation of
  connected {Shimura} varieties}, Selecta Math. Soviet. \textbf{3} (1983/84),
  no.~1, 3--39.

\bibitem[BS91]{BS91}
M.~Burger and P.~Sarnak, \emph{Ramanujan duals. {II}}, Invent. Math.
  \textbf{106} (1991), no.~1, 1--11.

\bibitem[Clo03]{Clo03}
L.~Clozel, \emph{D\'emonstration de la conjecture $\tau$}, Invent. Math.
  \textbf{151} (2003), no.~2, 297--328.

\bibitem[Del79]{Del79}
P.~Deligne, \emph{Vari\'et\'es de Shimura: interpr\'etation modulaire, et
  techniques de construction de mod\`eles canoniques},
  In: ``Automorphic forms, representations and $L$-functions (Part~2)'', Proc. Sympos. Pure Math.,
  XXXIII, AMS, Providence, R.I., 1979, 247--289.

\bibitem[DO16]{DO16}
C.~Daw and M.~Orr, \emph{Heights of pre-special points of Shimura varieties},
  Math. Ann. \textbf{365} (2016), no.~3-4, 1305--1357.
  
\bibitem[DO]{DO20}
C.~Daw and M.~Orr, \emph{Quantitative reduction theory and unlikely intersections},
  preprint, available at \href{arxiv.org/abs/1911.05618}{arXiv:1911.05618v2}. 

\bibitem[DR18]{DR}
C.~Daw and J.~Ren, \emph{Applications of the hyperbolic Ax--Schanuel
  conjecture}, Compos. Math. \textbf{154} (2018), 1843--1888.

\bibitem[DS05]{DS05}
F.~Diamond and J.~Shurman, ``A first course in modular forms'', Grad.
  Texts in Math., vol. 228, Springer-Verlag, New York, 2005.

\bibitem[Fal83]{Fal83}
G.~Faltings, \emph{Endlichkeitss\"atze f\"ur abelsche Variet\"aten \"uber
  Zahlk\"orpern}, Invent. Math. \textbf{73} (1983), no.~3, 349--366.

\bibitem[FC90]{FC90}
G.~Faltings and C.-L. Chai, ``Degeneration of abelian varieties'',
  Ergeb. Math. Grenzgeb. (3), vol.~22,
  Springer-Verlag, Berlin, 1990, with an appendix by D. Mumford.

\bibitem[HP12]{HP12}
P.~Habegger and J.~Pila, \emph{Some unlikely intersections beyond
  Andr\'e--Oort}, Compos. Math. \textbf{148} (2012), no.~1, 1--27.

\bibitem[HP16]{HP16}
P.~Habegger and J.~Pila, \emph{O-minimality and certain atypical intersections}, Ann. Sci.
  \'Ec. Norm. Sup\'er. (4) \textbf{49} (2016), no.~4, 813--858.

\bibitem[Huy16]{Huy16}
D.~Huybrechts, ``Lectures on K3 surfaces'', Cambridge Stud. Adv.
  Math., vol. 158, Cambridge University Press, Cambridge, 2016.

\bibitem[Jar14]{Jar14}
F.~Jarvis, ``Algebraic number theory'', Springer Undergrad. Math.
  Ser., Springer, 2014.

\bibitem[Kow]{Kow}
E.~Kowalski, ``An introduction to expander graphs'', Online notes, available
  at \url{https://people.math.ethz.ch/~kowalski/}.

\bibitem[Len87]{Len87}
H.~W. Lenstra, Jr., \emph{Factoring integers with elliptic curves}, Ann. of
  Math. (2) \textbf{126} (1987), no.~3, 649--673.

\bibitem[LM16]{LM16}
H.~Li and G.~A. Margulis, \emph{Effective estimates on integral quadratic
  forms: Masser's conjecture, generators of orthogonal groups, and bounds in
  reduction theory}, Geom. Funct. Anal. \textbf{26} (2016), no.~3, 874--908.

\bibitem[LZ89]{LZ89}
A.~Lubotzky and R.~J. Zimmer, \emph{Variants of Kazhdan's property for
  subgroups of semisimple groups}, Israel J. Math. \textbf{66} (1989), no.~1-3,
  289--299.

\bibitem[LZ03]{LZ03}
A.~Lubotzky and A.~\.{Z}uk, ``On property ($\tau$)'', Preliminary version,
  available at \url{http://www.ma.huji.ac.il/\~alexlub/}, 2003.

\bibitem[Mas]{Masser}
D.W.~Masser, \emph{Semi--effective Andr\'e--Oort
for real multiplication curves in the Siegel threefold}, Preliminary version sent via private communication.

\bibitem[Mil83]{Mil83}
J.~S. Milne, \emph{The action of an automorphism of ${\bf C}$ on a Shimura
  variety and its special points}, In: ``Arithmetic and geometry, Vol. I'',
  M. Artin and J. Tate (eds.), Progr. Math., vol.~35,
  Birkh\"{a}user Boston, Boston, MA, 1983, 239--265.

\bibitem[Mil05]{Mil05}
J.~S. Milne, \emph{Introduction to Shimura varieties}, In: ``Harmonic analysis, the
  trace formula, and Shimura varieties'', J. Arthur, D. Ellwood and R. Kottwitz (eds.),
  Clay Math. Proc., vol.~4, Amer. Math. Soc., Providence, RI, 2005, 265--378.

\bibitem[MS82]{MS82b}
J.~S. Milne and K.~Y. Shih, \emph{Conjugates of Shimura varieties}, In: ``Hodge
  cycles, motives, and Shimura varieties'', Lecture Notes in Math., vol.
  900, Springer-Verlag, 1982, 280--356.

\bibitem[Mum74]{Mum74}
D.~Mumford, ``Abelian varieties'', second ed., Oxford University Press,
  1974.

\bibitem[MW93]{MW93}
D.~Masser and G.~W\"ustholz, \emph{Isogeny estimates for abelian varieties,
  and finiteness theorems}, Ann. of Math. (2) \textbf{137} (1993), no.~3,
  459--472.

\bibitem[MZ99]{MZ99}
B.~J.~J. Moonen and Yu.~G. Zarhin, \emph{Hodge classes on abelian varieties of
  low dimension}, Math. Ann. \textbf{315} (1999), no.~4, 711--733.

\bibitem[Orr18]{Orr18}
M.~Orr, \emph{Height bounds and the Siegel property}, Algebra Number
  Theory \textbf{12} (2018), no.~2, 455--478.

\bibitem[Orr19]{OrrSV}
M.~Orr, \emph{Galois conjugates of special points and special subvarieties in
  Shimura varieties}, J. Inst. Math. Jussieu, published online 2019-10-30.

\bibitem[Orr]{OrrUI}
M.~Orr, \emph{Unlikely intersections with Hecke translates of a special
  subvariety}, to appear in J. Eur. Math. Soc. (JEMS), available at
  \href{arxiv.org/abs/1710.04092v2}{arXiv:1710.04092v2}.

\bibitem[PT13]{PT13}
J.~Pila and J.~Tsimerman, \emph{The Andr\'e--Oort conjecture for the moduli
  space of abelian surfaces}, Compos. Math. \textbf{149} (2013), no.~2,
  204--216.

\bibitem[PT14]{PT:Ag}
J.~Pila and J.~Tsimerman, \emph{Ax--Lindemann for $\mathcal{A}_g$}, Ann. of Math. (2)
  \textbf{179} (2014), no.~2, 659--681.

\bibitem[Tsi12]{Tsi12}
J.~Tsimerman, \emph{Brauer--Siegel for arithmetic tori and lower bounds for
  Galois orbits of special points}, J. Amer. Math. Soc. \textbf{25} (2012),
  no.~4, 1091--1117.

\bibitem[Tsi18]{Tsim:AO}
J.~Tsimerman, \emph{The Andr\'e--Oort conjecture for $\mathcal{A}_g$}, Ann. of
  Math. (2) \textbf{187} (2018), no.~2, 379--390.

\bibitem[Ull07]{Ull07}
E.~Ullmo, \emph{Equidistribution de sous-vari\'et\'es sp\'eciales II}, J.
  Reine Angew. Math. \textbf{606} (2007), 193--216.

\bibitem[UY15]{UY15}
E.~Ullmo and A.~Yafaev, \emph{Nombre de classes des tores de multiplication
  complexe et bornes inf\'erieures pour les orbites galoisiennes de points
  sp\'eciaux}, Bull. Soc. Math. France \textbf{143} (2015), no.~1, 197--228.

\bibitem[Vas68]{Vas68}
L.~N. Vaserstein, \emph{On groups possessing property T}, Funktsional. Anal.
  i Prilozhen. \textbf{2} (1968), no.~2, 86 (English translation: Funct. Anal.
  Appl. {\bf 2} (1968), no.~2, 174).

\end{thebibliography}
\end{document}